\documentclass[10pt]{article}
\usepackage{amsmath,amssymb,url,graphicx,color,array,amsthm}
\usepackage{amsfonts,dsfont}
\usepackage{rotating}
\usepackage{fancyvrb}
\usepackage[ruled]{algorithm2e}
\usepackage[margin=1in]{geometry}
\usepackage[x11names]{xcolor}
\usepackage{hyperref}[6.83]
\hypersetup{
 colorlinks = true,
 allcolors = Blue3,
 urlcolor = Red3,
}
\usepackage[numbers]{natbib}
\usepackage{url}
\usepackage{enumitem}
\usepackage{cleveref}
\setlist[description]{style=multiline,leftmargin=2.3em,topsep=3mm,itemsep=0mm}
\setlist[itemize]{style=standard,leftmargin=2em,topsep=3mm,parsep=0mm,itemsep=0mm}

\def\[{\begin{equation}}
\def\]{\end{equation}}

\def\cK{{\mathcal K}}
\def\cC{{\mathcal C}}
\def\cD{{\mathcal D}}
\def\cL{{\mathcal L}}
\def\cT{{\mathcal T}}
\def\cH{{\mathcal H}}
\def\cU{{\mathcal U}}
\def\cG{{\mathcal G}}

\def\cQ{{\mathcal Q}}
\def\cM{{\mathcal M}}
\def\cW{{\mathcal W}}
\def\cS{{\mathcal S}}
\def\cE{{\mathcal E}}
\def\cF{{\mathcal F}}

\def\cQ{{\mathcal Q}}
\def\cJ{{\mathcal J}}
\def\cX{{\mathcal X}}
\def\cY{{\mathcal Y}}
\def\cZ{{\mathcal Z}}

\def\cN{{\mathcal N}}
\def\cP{{\mathcal P}}
\def\Pii{{\Pi}}
\def\ri{{\sf ri}}
\def \inti {{\sf int}}

\def\gph{{\sf gph}}
\def\dist{{\sf dist}}
\def\what{ \widehat }
\def\Re{{\mathds R}}
\def\prox{{\sf prox}}
\def\spa{{\sf span}}
\def\lin{{\sf {lin}}}
\def\range{{\sf {rge}}}

\def\ker{{\sf {ker}}}

\def\max {{\sf max}}
\def\min {{\sf min}}

\DeclareMathOperator*{\argmin}{arg\,min}
\newcommand{\longsetto}[1]{\mathop{\longrightarrow}\limits^{#1}}

\renewcommand{\Re}{{\mathds R}}
\newlength{\myparboxwidth}
\newtheorem{theorem}{Theorem}[section]
\newtheorem{lemma}{Lemma}[section]
\newtheorem{assumption}{Assumption}[section]
\newtheorem{proposition}{Proposition}[section]
\newtheorem{corollary}{Corollary}[section]

\newtheorem{remark}{Remark}[section]

\newtheorem{definition}{Definition}[section]

\title{\Large \bf 
Two Typical Implementable Semismooth* Newton Methods for Generalized Equations are G-Semismooth Newton Methods\thanks{The research of the first author is supported by
the National Key R\&D Program of China [No. 2021YFA1001300],
the National Natural Science Foundation of China [No. 12271150],
the Natural Science Foundation of Hunan Province [No. 2023JJ10001],
the Science and Technology Innovation Program of Hunan Province [No. 2022RC1190],
and
the research of the second author is supported by
the Hong Kong RGC Senior Research Fellow Scheme [No. SRFS2223-5S02] and the GRF Grants [Nos. 15307822 and 15307523].}}

\author{Liang Chen\thanks{School of Mathematics, Hunan University, Changsha, China (\href{mailto:chl@hnu.edu.cn}{chl@hnu.edu.cn}).}
\and
Defeng Sun\thanks{Department of Applied Mathematics, The Hong Kong Polytechnic University, Hung Hom, Hong Kong (\href{mailto:defeng.sun@polyu.edu.hk}{defeng.sun@polyu.edu.hk}).}
\and
Wangyongquan Zhang\thanks{Department of Applied Mathematics, The Hong Kong Polytechnic University, Hung Hom, Hong Kong (\href{mailto:wangyongquan.zhang@connect.polyu.hk}{wangyongquan.zhang@connect.polyu.hk}). }}

\date{\today}
\begin{document}
\maketitle

\begin{abstract}
Semismooth* Newton methods have been proposed in recent years targeting multi-valued inclusion problems and have been successfully implemented to deal with several concrete generalized equations.
In this paper, we show that two typical implementations of them that are available are exactly the applications of G-semismooth Newton methods for solving nonsmooth equations localized from these generalized equations.
This new understanding expands the breadth of G-semismooth Newton methods in theory,
 results in a few interesting problems regarding the two categories of nonsmooth Newton methods, and more importantly,
 provides informative observations in facilitating
 the design and implementation of practical Newton-type algorithms for solving generalized equations.


\bigskip\noindent
{\bf MSC2000 subject classification}: {49J52; 49J53; 90C31; 90C33; 49M15}
\end{abstract}

\section{Introduction} 
Starting from the seminal work of Kummer \cite{Kummer1},
Newton-type methods for solving nonsmooth equations have evolved for several decades.
The literature on this topic is abundant; one may refer to
\cite{Pang1,Kummer2,qisun,Robinson,Hoheisel12,Hoheiselerr,Izmailov15}
and references therein.
Nowadays, nonsmooth Newton methods have been heavily incorporated into efficient numerical optimization software for large-scale optimization problems \cite{sdpnal,sdpnalp,suitlasso}.
For generalized equations (GEs), nonsmooth Newton methods have also been extensively studied.
In the pioneering works \cite{Josephy1,Josephy2} of this field, Josephy considered the GEs in the form of
\begin{equation*}
0\in H(x)+\Theta(x),
\end{equation*}
where
$H:\cX \to \cY$ is a { single-valued} function,
$\Theta:\cX \rightrightarrows \cY$ is a multifunction, and $\cX$ and $\cY$ are finite-dimensional real Hilbert spaces each endowed with an inner product and its induced norm.
Studies in this direction include
\cite{Bonans94,Dontchev961,Dontchev96,Fischer99,solodov02,Geoffroy04,Geoffroy05,IS10,Dontchev10,cibulka11,gaydu13,aa2014,IS14,Aragon15,Izmailov15, Ferreira15,Ferreira17,Ferreira18,Oliveira19}, to name only a few.
In these algorithms, the single-valued part $H$ is linearized while the multi-valued part $\Theta$ is kept.
So, the subproblems are linearized generalized equations.
Alternatively, the constructions of Newton-type methods for nonsmooth GEs by approximating the multi-valued part $\Theta$ have been investigated in \cite{Aze95,Dias12,Hoheisel12,gaydu13,KK18}. 
 
Significant progress has also been made in Newton-type methods with subproblems being coderivative (or graphical derivative) inclusions \cite{mordukhovich2021generalized,mordukhovich2024second,Khanh1,Khanh2,Khanh3,Khanh4,Aragon24}.

Recently, a semismooth* Newton method was proposed in \cite{Gfrerer2021}
for solving GEs in the form of 
\begin{equation}
\label{EqGE}
 0\in \Phi(x),
\end{equation}
where $\Phi:\cX \rightrightarrows \cY$ is a set-valued mapping with closed graph, and was extended in \cite{gfrerer2022local} using subspace containing derivatives (SCD). 
A particular property of these semismooth* Newton methods is that the subproblems involved are linear systems of equations, which is different from the methods mentioned above for solving (\ref{EqGE}). 
Moreover, 
since an inequality involving the target solution should be properly fulfilled in the ``approximation step'' at each iteration of these algorithms, 
specific implementations of these algorithms should be elaborately designed to make them practical.
 
These implementations are rarely available in the literature.
However, the two executable representatives of them, even seemingly sophisticated, are of practical value. 
 
Specifically, on the one hand, the semismooth* Newton method in \cite{Gfrerer2021} for \eqref{EqGE}
was implemented in \cite[Section 5]{Gfrerer2021} for solving the GE
\[
\label{ssnge}
0\in F(x)+\nabla G(x) N_D(G(x)),
\]
where $F:\Re^n\to\Re^n$ is continuously differentiable, $G:\Re^n\to\Re^s$ is twice continuously differentiable,
$D\subseteq \Re^s$ is a convex polyhedral set,
$\nabla G(x)$ denotes the adjoint of the Jacobian operator $G'(x)$,
and $N_D(\cdot)$ denotes the normal cone mapping used in convex analysis.
As mentioned in \cite{Gfrerer2021}, the GE \eqref{ssnge} arises frequently in optimization and equilibrium models, and it is equivalent to the GE given by $0\in F(x)+ N_{G^{-1}(D)}(x)$ under certain constraint qualifications \cite[Proposition 2.1]{calmref}. 
On the other hand, in \cite{Gfrerer222}, the SCD semismooth* Newton method was applied to the GE
\[
\label{gecp}
0\in F(x)+\partial q(x),
\]
where $\partial q$ is the subdifferential mapping of a given closed proper convex function $q:\Re^n\to(-\infty,\infty]$. 
Such an implementation was further extended in \cite{gf23} to a more general class of GEs.

The GEs \eqref{ssnge} and \eqref{gecp} are of significant importance for taking a closer look at these semismooth* Newton methods.
Since the corresponding subproblems for computing the Newton directions are linear systems of equations, which are akin to the subproblems of semismooth Newton methods,
it is natural to ask whether these implementations admit a connection to the existing Newton-type methods for solving nonsmooth equations.
Note that, when $\Phi$ in \eqref{EqGE} is single-valued and locally Lipschitz continuous around a point $\bar x$, it is easy to see that the semismooth* property of $\Phi$ at $\bar x$ is exactly G-semismoothness (semismoothness in the sense of Gowda \cite{Gowda2004}) of $\Phi$ at $\bar x$, { for example, by} \cite[Proposition 3.7]{Gfrerer2021}.
Moreover, when solving a locally Lipschitz continuous equation, the relationship between the semismooth* Newton method and the G-semismooth Newton method of Kummer \cite{Kummer1} (c.f. \Cref{sec:gss} for details) has been discussed in \cite[Sect. 4]{Gfrerer2021}.
For solving GEs beyond nonsmooth equations,
the relationship between the two types of nonsmooth Newton methods is unknown.

In this paper, by reformulating the GEs \eqref{ssnge} and \eqref{gecp} as nonsmooth equations, which are proved to be locally Lipschitz continuous,
we show that the corresponding practical implementations of semismooth* Newton methods
are exactly the applications of G-semismooth Newton methods.
Specifically, we show that the algorithm implemented in \cite[Section 5]{Gfrerer2021} for solving \eqref{ssnge} is an application of a G-semismooth Newton method for solving an implicitly defined equation.
Furthermore, for the algorithm in \cite{Gfrerer222} for solving \eqref{gecp}, we take the proximal residual mapping as the Lipschitz continuous localization of \eqref{gecp}, and show that the implemented SCD semismooth* Newton method is also an application of a G-semismooth Newton method.
Additionally, we show that the conditions for ensuring the convergence of these semismooth* Newton methods are sufficient for the corresponding applications of the G-semismooth Newton methods.
Therefore, one can conclude that these implementable semismooth* Newton methods are G-semismooth Newton methods. 
This leads to a concrete foundation for comprehending semismooth* Newton methods and is beneficial for developing practical nonsmooth Newton methods for solving GEs, especially considering globalization.
 
Here, we emphasize that we focus on the local convergence properties. 
For globalizing the G-semismooth Newton method, one may refer to the (inexact) smoothing Newton methods studied in \cite{QiSunZhou2000,GaoSun2009} (note that though semismoothness was used in the cited two references, a quick examination reveals that G-semismoothness is sufficient for convergence and rate of convergence analysis). 
In addition, we only consider GEs or nonsmooth equations instead of $\cC^{1,1}$ optimization problems, for which traditional globalized G-semismooth Newton methods generally require the directional differentiability of the gradient mapping.
This requirement can be removed by involving the Lipschitz constant of the gradient mapping and a modulus for local stability in the line search as in the recent coderivative-based nonsmooth Newton methods for $\cC^{1,1}$ optimization problems or composite optimization problems with $\cC^{1,1}$ envelopes \cite{mordukhovich2024second,Khanh2,Khanh4}. 
Most recently, the issue of globalizing the semismooth* Newton method for nonconvex composite optimization problems has also been discussed \cite{gfrerer2025}.

The remaining parts of this paper are organized as follows.
In \Cref{sectpre}, we collect some basic results in variational analysis, and briefly introduce the G-semismooth Newton method. 
 
In \Cref{secimplementation}, the implementable semismooth* Newton methods in \cite[Section 5]{Gfrerer2021} and \cite{Gfrerer222} are introduced, together with some intermediate results which are necessary for further discussions.
In \Cref{secinv} and \Cref{secscd}, we show that these executable implementations of semismooth* Newton methods are applications of G-semismooth Newton methods for solving nonsmooth equations involving locally Lipschitz continuous functions. This constitutes the main contribution of this work.
We conclude our paper in \Cref{secconlu}.
 
\section{Preliminaries}
\label{sectpre}
This section presents the definitions and necessary tools from variational analysis \cite{rock1998,mordukhovich2006variational,mordukhovich2024second}. 
It also provides preliminary results and reviews the G-semismooth Newton method.

\subsection{Basic variational analysis}
Let $\cX$ and $\cY$ be two finite-dimensional real Hilbert spaces each equipped with an inner product $\langle\cdot,\cdot\rangle$ and its induced norm $\|\cdot\|$.
For any $x\in\cX$ and $\delta>0$,
${\mathbb B}_\delta(x)$ denotes the closed ball centered at $x$
with radius $\delta$, and ${\mathbb B}_\cX$ and ${\mathbb B}_\cY$ are the unit balls in $\cX$ and $\cY$, respectively.
Moreover, $[x]$ denotes the subspace spanned by the given vector $x\in\cX$.
The notation $(\cdot;\cdot )$ means two vectors or linear operators are stacked symbolically in column order.
For a subspace $\cX_0$ of $\cX$, we use $\cX_0^\perp$ to denote its orthogonal complement in $\cX$.
We use ${\mathbb L}(\cX,\cY)$ to represent the space of all linear operators from $\cX$ to $\cY$, and write ${\mathbb L}(\cX)\equiv {\mathbb L}(\cX,\cX)$ for convenience.
For an arbitrary linear operator $V$, we use $\range V$ to denote its range space, $\ker V$ to denote its null space.
If V is a matrix, we use $V^\top$ to denote its transpose.

For a nonempty set $C\subseteq \cX$, we use
$\ri C$ and $\inti C$ to denote the relative interior and interior of $C$, respectively.
The lineality space of $C$, denoted by $\lin C$, is the largest linear subspace contained in $C$.
Meanwhile, we use $\spa C$ to denote the smallest linear subspace that contains $C$.
When $C$ is locally closed at $\bar{x} \in C$, the {contingent (Bouligand) cone} $T_{C}(\bar{x})$, the {regular (Fr\'{e}chet) normal cone} $\widehat{N}_{C}(\bar{x})$ and 
the limiting (Mordukhovich) normal cone $N_{C}(\bar{x})$ to $C$ at $\bar{x}$ are defined, respectively, by
$$
\begin{array}{ll}
T_{C}(\bar{x}):=\limsup\limits_{t {\searrow} 0} \frac{C-\bar{x}}{t}, 
\quad
\widehat{N}_{C}(\bar{x}):=(T_{C}(\bar{x}))^{\circ},
\quad \mbox{and}\quad 
N_{C}(\bar{x}):=\limsup\limits_{{x \rightarrow \bar{x}}, x\in C} \widehat{N}_{C}(x).
\end{array}
$$ 
Furthermore, 
$\cK_C(\bar{x},\, d):=T_C(\bar{x})\cap[d]^\perp$ is the critical cone to $C$ at $\bar x\in C$ with respect to 
$d\in\what{N}_{C}(\bar{x})$. 

If $K\subseteq\cX$ is a closed convex cone, we use $K^\circ$ to denote its polar, i.e.,
$K^\circ:=\{x\in\cX \mid \langle x,x'\rangle \le 0\, \forall x'\in K\}$.
In this case, one has $\lin K=K\cap(-K)$ and $\spa K=K+(-K)$.
Moreover, it holds
$(\lin K)^{\perp}=\spa K^{\circ}$ and
$(\spa K)^{\perp}=\lin K^{\circ}$.

For a set-valued mapping $\Phi: \cX \rightrightarrows \cY$,
we use $\gph \Phi$ to denote its graph in $\cX\times\cY$.
The mapping $\Phi$ is called outer semicontinuous at $x$
if for any $\varepsilon>0$ there exists $\delta>0$ such that $\Phi(x')\subseteq\Phi(x)+\varepsilon {\mathbb B}_\cY$ holds for all $x'\in x+\delta {\mathbb B}_\cX$.
When $\gph\Phi$ is (locally) closed, the {regular (Fr\'echet) coderivative} and the {limiting (Mordukhovich) coderivative} of $\Phi$ at $(\bar{x},\bar{y} )$ are the multifunctions $\widehat D^\ast \Phi(\bar{x},\bar{y} ): \cY\rightrightarrows\cX$
and $D^\ast \Phi(\bar{x},\bar{y} ): \cY\rightrightarrows\cX$, respectively, such that 
\begin{equation*}
\begin{cases}
\widehat D^\ast \Phi(\bar{x},\bar{y} )(v^\ast):=\{u^\ast\in \cX \mid (u^\ast ;- v^\ast)\in \widehat
N_{\gph \Phi}(\bar{x} ;\bar{y} )\}\quad \forall v^\ast\in \cY,
\\
D^\ast \Phi(\bar{x},\bar{y} )(v^\ast):=\{u^\ast\in \cX \mid (u^\ast ;- v^\ast)\in N_{\gph \Phi}(\bar{x} ;\bar{y} )\}\quad\forall
v^\ast\in \cY.
\end{cases}
\end{equation*}
If $\Phi$ is single-valued, one can write the two coderivatives as $\widehat D^*\Phi(\bar{x})$ and
$D^*\Phi(\bar{x})$ for simplicity.
If $\Phi$ is Fr\'echet differentiable at $\bar{x}$, by \cite[Theorem 1.38]{mordukhovich2006variational} one has
$\widehat D^*\Phi(\bar{x})(v^*)=\{\nabla \Phi(\bar{x}) v^*\}$, where $\nabla \Phi(\bar{x})$ is the adjoint of the Fr\'echet derivative $\Phi'(\bar x)$.
If $\Phi$ is strictly differentiable at $\bar{x}$, one also has $D^*\Phi(\bar{x})(v^*)=\{\nabla \Phi(\bar{x}) v^*\}$.
If $\Phi$ is Lipschitz continuous in an open neighborhood $\Omega$ of $\bar{x}$, from
Rademacher's theorem \cite{Rademacher} we know $\Phi$ is almost everywhere Fr\'echet differentiable in $\Omega$.
In this case,
the Bouligand subdifferential of $\Phi(\cdot)$ at $\bar x$ is defined by
\[
\label{bouligand}
\partial_{B}\Phi(\bar x):=
\big\{\lim_{k\to\infty}
\Phi'(x^{(k)})
\mid \Phi\text{ is differentiable at }
x^{(k)},\, x^{(k)}\rightarrow \bar x
\big\}
\]
and Clarke's generalized Jacobian of $\Phi$ at $\bar x$ is defined by $\partial \Phi(\bar x):={\rm conv}\partial_B\Phi(\bar x)$, i.e., by taking the convex hull of the Bouligand subdifferential.

In \cite{Gfrerer2021}, a generalization of the coderivatives was introduced. 
Specifically,
for $\Phi:\cX\rightrightarrows\cY$ with closed graph,
one can let
$\widehat\cD^* \Phi:\gph \Phi\to (\cY\rightrightarrows\cX)$ be a mapping such that for every pair $(x ;y)$ in $\gph \Phi$, the set $\gph \widehat\cD^* \Phi(x,y)$ is a cone. One can define the associated limiting mapping
$\cD^* \Phi:\gph \Phi\to (\cY\rightrightarrows\cX)$
by
$$
\begin{array}{l}
 \gph \cD^* \Phi(x,y):
=\limsup\limits_{(x' ;y')\longsetto{\gph \Phi}(x ;y)}\gph\widehat\cD^* \Phi(x',y').
\end{array}
$$
Here, $\widehat\cD^*$ and $\cD^*$ serve as the generalizations of the regular and limiting coderivatives $\widehat D^*$ and $D^*$, respectively. 
In \cite{Gfrerer2021}, the notion of semismoothness* was originally proposed for sets and is equivalent to the semismoothness of sets in \cite[Definition 2.3]{Henrion}.
The following definition of generalized semismoothness* comes from its application and generalization to $\gph\Phi$.

\begin{definition}[{\cite[Definition 4.8]{Gfrerer2021}}]
\label{defsst}
Let $\Phi:\cX \rightrightarrows\cY$ and $(\bar{x} ;\bar{y})\in\gph \Phi$, which is nonempty and closed.
Then, $\Phi$ is called semismooth* at
$(\bar{x},\bar{y})$ with respect to $ {\cD}^* \Phi$ if for every $\epsilon>0$ there is some $\delta>0$ such that
$$
\vert \langle x^*,x-\bar{x}\rangle-\langle y^*,y-\bar{y}\rangle\vert\leq \epsilon
\|(x;y)-(\bar{x};\bar{y})\|\|(x^*;y^*)\|\quad 
\forall(x;y)\in \mathbb{B}_\delta(\bar{x};\bar{y}), \forall
(y^*;x^*)\in\gph\widehat\cD^*\Phi(x,y).
$$
\end{definition}

Finally, we discuss proximal mappings and projections. 
For a maximal monotone mapping $\cM:\cX\rightrightarrows\cX$, the corresponding proximal mapping is defined by
$
\prox_{\lambda \cM}:=(I+\lambda \cM)^{-1}$, $\lambda>0$,
which is single-valued and Lipschitz continuous with unit Lipschitz constant, where $I$ represents the identity operator.
Given a closed proper convex function $q: \cX \to (-\infty,+\infty]$ and a parameter $\lambda >0$, its subdifferential mapping $\partial q$ is always maximal monotone. Taking $\cM\equiv \partial q$, it is easy to see that
\[
\label{proxdef}
\prox_{\lambda \cM}(x)=\cP_{\lambda q}(x):=\argmin_z\Big\{q(z)+\frac{1}{2\lambda}\|z-x\|^2\Big\} \quad \forall x\in \cX.
\]
Therefore, for a nonempty closed convex subset $C\subseteq \cX$ with $\delta_C$ being its indicator function, the projection mapping (with respect to $\|\cdot\|$) can be defined by
$\Pii_C(x):=\cP_{\lambda\delta_C}(x)$ for any $\lambda>0$.
The following lemma on the Bouligand subdifferential of the projection mapping onto a convex polyhedral set is necessary for subsequent discussions.
\begin{lemma}
\label{partialbpiq}
Let $\cQ\subset\Re^l$ be a nonempty polyhedral convex set.
Then for a given $\mu\in \Re^l$,
\begin{equation*}\partial_B\Pii_\cQ(\mu)
=\big\{
\Pii_{\spa \cG}(\cdot)
\mid
\cG \mbox{ is a face of } K(\mu):
=\mathcal{K}_\cQ(\Pii_\cQ(\mu),\, \mu-\Pii_\cQ(\mu))
\big\}.
\end{equation*}
\end{lemma}

\begin{proof}{Proof}
We know from \cite[Lemma 5(i)]{Pang1} and the definition of
$K(\mu)$
that for any $\Delta\mu\in \Re^l$
with $\|\Delta\mu\|$ being sufficiently small, it holds that
\[
\label{projeq}
\Pii_\cQ(\mu+\Delta\mu)= \Pii_\cQ(\mu)+
\Pii_{K(\mu)}(\Delta\mu).
\]
Moreover, if $\Pii_\cQ$ is differentiable at $\mu+\Delta\mu$, one has from \cite[Lemma 5(ii)]{Pang1} that
\begin{equation*}
\Pii_\cQ ' (\mu+\Delta\mu)=\Pii_{\lin K(\mu+\Delta\mu)} (\cdot)= \Pii_{\lin T_\cQ(\Pii_{\cQ}(\mu+\Delta\mu))} (\cdot),
\end{equation*}
where the last equality comes from (2.3) of \cite[Lemma 2.4]{Gfrerer2021}.
Note that $\lim\limits_{\Delta\mu\to 0}\Pii_{K(\mu)}(\Delta\mu)=0$.
Therefore, from \cite[Lemma 2.4]{Gfrerer2021} we know that for every $\Delta\mu$ with $\|\Delta\mu\|$ sufficiently small one has
${\lin} K(\mu+\Delta\mu)=\spa\cG$ with $\cG$ being a face of $K(\mu)$.
Consequently, one has
$\Pii_\cQ ' (\mu+\Delta\mu)=\Pii_{\spa \cG}(\cdot)$, so that
\begin{equation}
\label{partialin}
\partial_B\Pii_\cQ( \mu)\subseteq\{
\Pii_{\spa \cG}(\cdot)
\mid
\cG \mbox{ is a face of } K(\mu)
\}.
\end{equation}
On the other hand, let $\cG$ be an arbitrary face of $K(\mu)$.
It holds $\cG=K(\mu)\cap[\nu]^\perp$ for some $\nu\in K(\mu)^\circ$.
Since $\cG$ is a closed convex set, one has from \cite[Theorem 6.2]{rock1970} that ${\ri}\cG$ is nonempty.
Let $\tilde\mu\in {\ri}\cG\subseteq K(\mu)$ be fixed.
it holds that $T_\cG(\tilde\mu)={\spa}\cG$. 
Moreover, from Moreau's decomposition theorem \cite[Theorem 31.5]{rock1970} one can get
$\Pii_{K(\mu)}(\tilde\mu+\nu)=\tilde\mu$.
Then, by \cite[Lemma 5.1(i)]{Pang1} we know that for all $\Delta\mu\in\Re^l$ with $\|\Delta\mu\|$ sufficiently small, it holds that
$\Pii_{K(\mu)}(\tilde\mu+\nu+\Delta\mu)
=\tilde\mu+\Pii_{
T_{K(\mu)}(\tilde\mu)\cap [\nu]^{\perp}
}(\Delta\mu)
=\tilde\mu+\Pii_{{\spa}\cG}(\Delta\mu)$. 
Thus, $\Pii_{K(\mu)}(\cdot)$ is differentiable at $\tilde\mu+\nu$
with $\Pii'_{{K(\mu)}} (\tilde \mu+\nu)=\Pii_{{\spa}\cG}(\cdot)$.
Note that for any integer $k>0$, one has $\cG=K(\mu)\cap[\nu/k]^\perp$.
Meanwhile, as $\cG$ is a closed convex cone, one has $\tilde\mu/k\in{\ri}\cG$. Therefore,
$\Pii_{K(\mu)}(\cdot)$ is differentiable at $(\tilde\mu+\nu)/k$ for all $k>0$
with $\Pii_{K(\mu)}'(\tilde\mu/k+\nu/k)=\Pii_{{\spa}\cG}(\cdot)$.
Consequently, by \eqref{projeq} one has
\begin{equation*}
\Pii_\cQ'(\mu+( \tilde \mu/k+\nu/k))
=\Pii_{{K(\mu)}}'(\tilde \mu/k+\nu/k)=\Pii_{{\spa}\cG}(\cdot).
\end{equation*}
Taking limits in the above equality along with $k\to\infty$, one gets $\Pii_{{\spa}\cG}(\cdot)\in\partial_B \Pii_\cQ(\mu)$.
This, together with \eqref{partialin}, completes the proof of the lemma.
\end{proof}

\subsection{A G-semismooth Newton method}
\label{sec:gss}
 
The terminology ``G-semismoothness'' was coined in \cite{pang2003} for distinguishing the definition of semismoothness in \cite{Gowda2004} by Gowda from those in \cite{mifflin,qisun} involving directional differentiability.
Specifically, let $\Omega\subseteq\cX$ be an open set,
$H:\Omega\to\cY$ be a continuous function,
and $\cT:\Omega \rightrightarrows {\mathbb L}(\cX,\cY)$ be a set-valued mapping.
According to \cite[Defintion 2]{Gowda2004}, we say $H$ is called G-semismooth with respect to $\cT$ at $ x\in \Omega$ if for any $h\to 0$ and $V\in\cT( x+h)$ it holds that $H(x+h)-H( x)-V h=o(\|h\|)$.
When $H$ is locally Lipschitz continuous around ${x}$,
it is simply called G-semismooth at ${x}$ if $\cT$ is taken as $\partial H$, 
and this definition is invariant if $\partial H$ is replaced by $\partial_BH$ (cf. \cite[Section 2.3]{Gowda2004}).
The following G-semismooth Newton method, based on G-semismoothness, is a trivial inexact extension of \cite{Kummer1}.

\SetKwFunction{Union}{Union}\SetKwFunction{FindCompress}{FindCompress}
\SetKwInOut{Input}{Input}\SetKwInOut{Output}{Output}
\begin{algorithm}[H]
\caption{A G-semismooth Newton method for solving nonsmooth equations \label{alg:nm}}
\normalsize
\Input{$x^{(0)}\in\cX$, $ H:\Omega\subseteq\cX\to\cY$, $\cT: \cX\rightrightarrows {\mathbb L}(\cX,\cY)$, and $\varrho\ge 0$.}
\Output{$\{x^{(k)}\}$.}
\For{$k=0,1, \ldots$,}{
1. if $ H(x^{(k)})=0$, stop the algorithm;
\\[1mm]
2. select
$V_k\in{\mathbb L(\cX,\cY)}$ such that ${\dist}(V_k,\cT(x^{(k)}))\le \varrho \| H(x^{(k)})\|$,
compute $\Delta x^{(k)}$ via solving
$V_{k}\Delta x=- H(x^{(k)})$,
and obtain $x^{(k+1)}:= x^{(k)}+\Delta x^{(k)}$.
}
\end{algorithm} 

\begin{theorem}
\label{thmgss}
Let $\Omega\subseteq\cX$ be an open set.
Suppose that $ H:\Omega\to\cX$ is locally Lipschitz continuous (with modulus $\vartheta>0$) and G-semismooth with respect to
$\cT:\Omega\rightrightarrows{\mathbb L}(\cX,\cY)$ at $\bar x$ such that $ H(\bar x)=0$.
Assume that $\cT(\cdot)$ is compact valued and outer semicontinuous at $\bar x$,
and $V^{-1}$ exists for all $V\in \cT(\bar x)$.
Then there exists a neighborhood of $\bar x$ such that, for any $x^{(0)}$ in it, Algorithm \ref{alg:nm} either terminates in finitely many steps or generates an infinite sequence $\{x^{(k)}\}$ satisfying
$\|x^{(k+1)}-\bar x\| = o(\|x^{(k)}-\bar x\|)$ as $k\to\infty$.
\end{theorem}

\begin{proof}{Proof}
Note that 
$\| H(x)\|\le \vartheta\|x-\bar x\|$ for all $x$ sufficiently close to $\bar x$.
Meanwhile, the G-semismoothness of $ H$ at $\bar x$ implies that the multifunction $\cT(x)+ \vartheta\varrho\|x-\bar x\| {\mathbb B}_{{\mathbb L}(\cX,\cY)}$ is a Newton map (c.f. \cite{Kumer2000} or \cite[Definition 2]{KK18} for the definition) for $H$ at $\bar x$.
As $\cT(\cdot)$ is compact valued and outer semicontinuous at $\bar x$ and all $V\in\cT(\bar x)$ are nonsingular, the Newton-regularity condition \cite[Definition 3]{KK18} holds at $\bar x$.
So the convergence properties of Algorithm \ref{alg:nm} follow from \cite[Theorem 4]{KK18} (or \cite[Lemma 10.1]{KK2}).
 \end{proof}

\section{Implementable semismooth* Newton methods for GEs}
\label{secimplementation}
This section reviews the two typical semismooth* Newton methods which are implementable to concrete GEs. 

\subsection{A semismooth* Newton method for the GE \texorpdfstring{\eqref{ssnge}}{(2)}}
In \cite[Section 5]{Gfrerer2021}, the semismooth* Newton method \cite[Algorithm 3]{Gfrerer2021} for solving the GE \eqref{EqGE} was implemented to \eqref{ssnge} by introducing an auxiliary variable $d\in\Re^s$ and solving the equivalent problem
\[
\label{cge-equiv}
0\in \cH(x,d):=\begin{pmatrix}
F(x)+\nabla G(x) N_D(d)
\\
G(x)-d
\end{pmatrix}.
\]
Note that $\bar x$ solves \eqref{ssnge} if and only if $(\bar x, \bar d)=(\bar x, G(\bar x))$ solves \eqref{cge-equiv}.
For convenience, define the Lagrangian function
\begin{equation*}
\mathcal{L}_\lambda(x):=F(x)+\nabla G(x)\lambda
\quad
\forall (x;\lambda )\in\Re^n\times\Re^s.
\end{equation*}
For a given point $\hat z:=((\hat x, \hat d) ;(\hat p^*, G (\hat x)-\hat d))\in\gph \cH$, one can choose $\hat \lambda\in N_D(\hat d)$ such that $\hat p^*=\cL_{\hat \lambda}(\hat x)$.
Moreover, one can define for all $(p ;q^*)\in \Re^n\times\Re^s$ the mapping 
\[
\label{deftmap}
{\mathbb T} (\hat x,\hat d,\hat \lambda)(p,q^*):
=\big\{
(\nabla \cL_{\hat\lambda}(\hat x)p
+\nabla G(\hat x)q^*,d^*)
\mid
d^*+q^*\in \widehat D^*N_D(\hat d,\hat \lambda)(G'(\hat x)p)
\big\}.
\]
In this case, according to \cite[eq. (5.5)]{Gfrerer2021}, the regular coderivative of $\cH$ at $\hat z$ satisfies
\[
\label{codt}
\widehat D^* \cH\left(\hat z\right)(p,q^*)\subseteq
{\mathbb T}(\hat x, \hat d,\hat \lambda)(p,q^*).
\]
For implementing the semismooth* Newton method \cite[Algorithm 3]{Gfrerer2021} to \eqref{cge-equiv}, a mapping $\widehat{\cD}^* \cH$ that surrogates $\widehat{D}^* \cH$ has been specified based on \eqref{deftmap}.
Furthermore, by defining 
${\cD}^*\cH$ as the outer limit of
$\widehat{\cD}^* \cH$, it is know from \cite[Theorem 5.5]{Gfrerer2021} that 
the mapping $\cH$ in \eqref{cge-equiv} is semismooth* with respect to $\cD^*\cH$ at every point ${( (x, G( x));(0,0))}\in \gph \cH$. 
Then, the ``approximation step'' in \cite{Gfrerer2021} was given as the following algorithm.

\begin{algorithm}[H]
\caption{{An approximation step} \label{algo:approx}}
\Input{$x\in\Re^n$. }
\Output{$\hat u\in\Re^n$, $\hat x\in\Re^n$, $\hat\lambda\in\Re^s$, $\hat d\in\Re^s$, $\hat{p}^* \in\Re^n.$}
 1. compute
\[
\label{qp}
\hat u=\argmin_{u\in\Re^n}\Big\{\frac{1}{2}\|u\|^2+\langle F(x),u\rangle
\ {\big |}\
G(x)+G'(x)u\in D\Big\}
\]
together with a multiplier
$\hat\lambda \in N_D(G(x)+ G'(x)\hat u)$
satisfying
$\hat u + F(x) + \nabla G(x)\hat\lambda
=\hat u +\cL_{\hat\lambda}(x)=0$;
\\[1mm]
2. set
$\hat x :=x$,
$\hat d : =G(x)+G'(x)\hat u$,
$\hat{p}^*:=\mathcal{L}_{\hat{\lambda}}(\hat{x})$, and
$\hat{y}:=\big(\hat{p}^*, G(\hat{x})-\hat{d}\big)$.
\end{algorithm}
The semismooth* Newton method in \cite[Section 5]{Gfrerer2021} for solving \eqref{ssnge} is given as Algorithm \ref{algo:snm_GE}.

\begin{algorithm}[H]
\caption{An implementable semismooth$^*$ Newton method for solving \eqref{ssnge} \label{algo:snm_GE}}
\Input {$F:\Re^n\to\Re^n$,\ $G:\Re^n\to\Re^s$ and $x^{(0)}\in\Re^n.$}
\Output{$\{x^{(k)}\}.$}
\For{$k=0,1, \ldots$,}{
1. if $x^{(k)}$ solves \eqref{ssnge}, stop the algorithm;
\\[1mm]
2. run Algorithm \ref{algo:approx} with input $x^{(k)}$ to compute $\hat{\lambda}^{(k)},\hat{d}^{(k)}$ and $\mathcal{L}_{\hat{\lambda}^{(k)}}(\hat x^{(k)})$;
\\[1mm]
3. set ${\hat{l}}^{(k)}=\dim(\operatorname*{lin}T_{D}(\hat{d}^{(k)}))$ and compute an $s\times(s-\hat{l}^{(k)})$ matrix $\hat{W}^{(k)}$, whose
columns form a basis for $\spa N_{D}(\hat{d}^{(k)})$ and then an $n\times(n-(s-\hat{l}^{(k)}))$ matrix $\hat{Z}^{(k)}$, whose columns are an orthogonal basis for $\ker {(\hat{W}^{(k)^\top}} G'(x^{(k)}))$;
\\[1mm]
4. set $x^{(k+1)}:=x^{(k)}+\Delta {x}^{(k)}$ with the Newton direction $\Delta x^{(k)}$ being a solution to the linear system
\begin{equation}
\left\{
\label{iteration_scheme_snm_GE}
 \begin{array}{l}
\hat{Z}^{{(k)}^\top}\bigl(\mathcal{L}'_{\hat{\lambda}^{(k)}}(x^{(k)})\Delta x^{(k)}+\mathcal{L}_{\hat{\lambda}^{(k)}}(x^{(k)})\bigr)=0,
 \\
 \hat{W}^{(k)^\top }\bigl(G(x^{(k)})+G'(x^{(k)})\Delta x^{(k)}-\hat{d}^{(k)}\bigr)=0.
 \end{array}
 \right.
\end{equation}
}
\end{algorithm}
Recall that a point $(x,d)$ is called nondegenerate with modulus $\gamma>0$ to the GE \eqref{ssnge} if one has
$\|\nabla G(x)\lambda\|\ge\gamma\|\lambda\|$ for all $\lambda\in {\spa} N_D(d)$.
It is called nondegenerate if the above condition holds with some $\gamma>0$.
The following assumption {\cite[Assumption 1]{Gfrerer2021}} is essential for Algorithm \ref{algo:approx}.
\begin{assumption}
\label{as1}
$(\bar x,G(\bar x))$ is a nondegenerate solution to \eqref{cge-equiv} with modulus $\tilde\gamma>0$.
\end{assumption}

\begin{remark} \label{equi_form_nondegenerate}
The point $(x,d)$ is called nondegenerate if and only if
$G'( x)\Re^n+\lin T_D( d)=\Re^s$.
Moreover, from \cite[Remark 5.3 and Lemma 5.4]{Gfrerer2021} it holds
that for any $\hat z:=((\hat x, \hat d) ;(\hat p^*,G(\hat x)-\hat d))\in\gph \cH$ with $(\hat x,\hat d)$ being nondegenerate,
one has \eqref{codt} holds as equality and there exists only one $\lambda\in N_D(\hat d)$, denoted by $\hat\lambda(\hat x,\hat d,\hat p^*)$, such that $\hat p^*=F(\hat x)+\nabla G(\hat x)\lambda$.
\end{remark}

Then one has the following result (there is a typo in \cite[eq. (5.13)]{Gfrerer2021} we take the revised form).
\begin{proposition}[{\cite[Proposition 5.7]{Gfrerer2021}}]
\label{propkey}
Under Assumption \ref{as1}, there exists a positive radius $\omega$ and positive reals $\beta,\beta_u$, and $\beta_\lambda$ such that for all $x\in \mathbb{B}_{\omega}(\bar x)$, the quadratic program in \eqref{qp} is well-defined and admits a unique solution $\hat u$, and the output of Algorithm \ref{algo:approx} satisfies
\begin{equation*}
\begin{array}{ll}
\|\hat u\|\le\beta_u\|\hat x-\bar x\|,
\quad
\|((\hat x;\hat d);\hat y)-((\bar x; G(\bar x));(0;0))\|\le\beta\|\hat x-\bar x\|,
\ \mbox{and}\ 
\|\hat \lambda -\bar \lambda \|\le\beta_\lambda\|\hat x-\bar x\|,
\end{array}
\end{equation*}
 where $\bar \lambda$ is the unique multiplier for $\bar x$. 
 
Further, $(\hat{x},\hat{d})$ is nondegenerate with modulus $\tilde\gamma/2$ and $N_{D}(\hat{d})\subseteq N_{D}(G(\bar{x}))$.
\end{proposition}

The following assumption {\cite[Assumption 2]{Gfrerer2021}} provides a regularity condition that guarantees that the linear system \eqref{iteration_scheme_snm_GE} admits a unique solution (there is a typo in \cite{Gfrerer2021} and we use the corrected form here).

\begin{assumption}
\label{ass_2}
For any face $\cF$ of the critical cone
$\cK_D(G(\bar x),\bar \lambda)$ there is a matrix $Z_\cF$, whose columns form an orthogonal basis of $\{u\mid G'(\bar x )u\in {\spa} \cF\}$, such that the matrix $Z_\cF^\top\mathcal{L}'_{\bar\lambda}(\bar x)Z_\cF$ is nonsingular.
\end{assumption}

According to \cite[Theorem 5.12]{Gfrerer2021}, 
under Assumptions \ref{as1} and \ref{ass_2},
there exists a neighborhood $U$ of $\bar x$ such that for every starting point $x^{(0)} \in U$, Algorithm \ref{algo:snm_GE} either stops after finitely many iterations at a solution of \eqref{ssnge} or
produces a sequence $\{x^{(k)}\}$ converging superlinearly to $\bar x$.

\subsection{An SCD semismooth* Newton method for the GE \texorpdfstring{\eqref{gecp}}{(3)}}
In \cite{Gfrerer222}, 
the SCD semismooth* Newton method proposed in \cite{gfrerer2022local} 
for \eqref{EqGE} (with $\cX\equiv \cY$) was implemented to the GE \eqref{gecp}.
As was observed in \cite{gfrerer2022local}, when applying the semismooth* Newton methods to \eqref{gecp},
it is advantageous to work with linear subspaces $L\subseteq\cX\times\cX$ having the same dimension $n$ with $\cX$ and contained in the graph of the limiting coderivative at a certain point $(x ;y)\in\gph \Phi$, i.e.,
$L\subseteq \gph D^* \Phi(x,y)$.
Specifically, denote by $\mathbb{Z}_{n}$ the metric space of all n-dimensional subspaces of
$\cX\times\cX$ equipped with the metric $d_{\mathbb{Z}_{n}}(L_1,L_2):=\|\Pii_{L_1}-\Pii_{L_2}\|$,
where $\Pii_{L_{i}}$ is the projection operator on $L_{i}$, $ i = 1, 2$. Further, define
$L^*:= \{(-v^*;u^*) \,| \, (u^*;v^*) \in L^\perp\}$ for any $L\in \mathbb{Z}_{n}$.
According to \cite[Definition 3.3]{gfrerer2022local}, 
$\Phi:\cX \rightrightarrows \cX$ with closed graph is called graphically smooth of dimension $n$ at $(x,y)$, if $(x;y)\in \gph \Phi$ and $T_{\gph\Phi}(x ;y)\in \mathbb{Z}_n$.
Denote by $\mathbb{O}_\Phi$ the set of all points where $\Phi$ is graphically smooth of dimension $n$.
Then one can define for $\Phi$ the following set-valued mappings (from $\gph \Phi$ to $\mathbb{Z}_n$):
\begin{equation*}
\begin{array}{ll}
\widehat{\cS}^*_\Phi(x,y)
:=\begin{cases}\{(T_{\gph\Phi}(x ;y))^*\}& \mbox{if $(x,y)\in{\mathbb{O}}_\Phi$,}\\
 \emptyset&\mbox{else,}\end{cases}
&
\qquad
{\cS}^*_\Phi(x,y):=\limsup\limits_{(u ;v)\longsetto{{\gph \Phi}}(x ;y)} \widehat{\cS}^*_\Phi(u,v).
\end{array}
\end{equation*}
The following definition of SCD property also comes from \cite[Definition 3.3]{gfrerer2022local}.
\begin{definition}
\label{DefSCDProperty}
$\Phi$ is said to have the SCD property at $(x,y)$ if $(x;y)\in\gph \Phi$ and ${\cS}^*_\Phi(x,y)\not=\emptyset$.
It has the SCD property around $(x,y)$ if $(x;y)\in\gph \Phi$ and there is a neighborhood $\cN$ of $(x;y)$ such that $\Phi$ has the SCD property at every $(x';y')\in\gph \Phi\cap \cN$.
It is called an SCD mapping if $\Phi$ has the SCD property at every point $(x,y)$ such that $(x;y)\in \gph \Phi$.
\end{definition}
The following definition of SCD regularity was given in \cite[Definition 4.1]{gfrerer2022local}.
\begin{definition}
Define $\mathbb{Z}_n^{\rm reg}:=
\{L\in\mathbb{Z}_n\mid
(y^*;0)\in L\Rightarrow y^*=0\}$.
A mapping $\Phi:\cX \rightrightarrows \cX$ is called SCD regular around {$(x,y)$ if $(x;y)\in\gph \Phi$, }$\Phi$ has the SCD property around $(x,y)$
and
$\cS^*_\Phi(x,y)\subseteq \mathbb{Z}_n^{\rm reg}$.
Moreover, the modulus of SCD regularity of $\Phi$ around $(x,y)$ is defined by
\begin{equation*}
{\rm scd\,reg\;}\Phi(x,y):=\sup\{\|y^*\|{\,\mid\,} (y^*;x^*)\in L, L\in \cS^*_\Phi(x,y), \|x^*\|\leq 1\}.
\end{equation*}
\end{definition}
According to \cite[Lemma 3.7]{gfrerer2022local}, the SCD property was coined because for any subspace $L\in \cS^*_\Phi(x,y)$ one has $L\subseteq \gph D^*\Phi(x,y)$.
Moreover, based on this property, the semismoothness$^*$ in \Cref{defsst} can be extended to the following SCD semismoothness$^*$.
\begin{definition}[{\cite[Definition 5.1]{gfrerer2022local}}]
Let $\Phi:\cX \rightrightarrows\cX$ and $(\bar{x};\bar{y})\in\gph \Phi$, which is nonempty and closed.
Then, $\Phi$ is called {SCD semismooth*} at
$(\bar{x},\bar{y})$ if $\Phi$ has the SCD property around $(\bar x,\bar y)$ and for every $\epsilon>0$ there exists some $\delta>0$ such that
\begin{equation*} 
\begin{array}{ll}
\vert \langle x^*,x-\bar{x}\rangle-\langle y^*,y-\bar{y}\rangle\vert\leq \epsilon
\|(x;y)-(\bar{x};\bar{y})\|\|(x^*;y^*)\|
\\[2mm]
\qquad
\forall(x;y)\in \mathbb{B}_\delta(\bar{x};\bar{y}),\quad \forall
(y^*;x^*)\in\bigcup {\cS^*_\Phi(x,y)}\subseteq \gph D^*\Phi(x,y).
\end{array}
\end{equation*}
\end{definition}
Note that the GE \eqref{gecp} can be solved equivalently via finding $(x;d)\in\Re^n\times\Re^n$ such that
\[
\label{GExd}
0\in \cJ (x,d):=\begin{pmatrix}F(x)+\partial q(d)
\\x-d
\end{pmatrix}.
\]
Given $\gamma>0$, define the proximal residual mapping
$u_\gamma:\Re^n\to\Re^n$ by
\begin{equation}
\label{proximal}
\begin{array}{ll}
u_\gamma (x):=\cP_{\gamma^{-1}q} \big(x- {\gamma}^{-1}F(x)\big)-x \quad \forall x\in \Re^n.
\end{array}
\end{equation}
In \cite[Sect. 5]{Gfrerer222}, the following implementable SCD semismooth* Newton method was proposed for \eqref{gecp}, which is an application of \cite[Algorithm 1]{gfrerer2022local} to \eqref{GExd}.

\begin{algorithm}[H]
\caption{An implementable SCD semismooth$^*$ Newton method for solving \eqref{gecp} \label{algo:ssnm_GE}}

\Input{$x^{(0)}\in\Re^n$, $F:\Re^n\to\Re^n$ and $q:\Re^n\to\bar{\Re}.$}
\Output{$\{x^{(k)}\}.$}
\For{$k=0,1, \ldots$,}{
1. if $0\in F(x^{(k)})+\partial q(x^{(k)})$, stop the algorithm;
\\[1mm]
2. select  
$\gamma^{(k)}>0$, and compute
$u^{(k)}:=u_{\gamma^{(k)}}(x^{(k)})$,
$\hat{d}^{(k)}:=x^{(k)}+u^{(k)}$ and
$\hat{d}^{*(k)}:=-\gamma^{(k)}u^{(k)}-F(x^{(k)})$;
\\[1mm]
3. select $(X^{*(k)},Y^{*(k)})$ with $\range(Y^{*(k)};X^{*(k)})\in{\cS}^*_{\partial q}\big(\hat{d}^{(k)},\hat{d}^{*(k)}\big)$, and compute the Newton direction $\Delta x^{(k)}$ from
\begin{equation}\label{iteration}
 {(Y^{*(k)}}^\top F'(x^{(k)})+{X^{*(k)}}^\top )\Delta x^{(k)}={( \gamma^{(k)}Y^{*(k)}}^\top +{X^{*(k)}}^\top )u^{(k)},
 \end{equation}
 and obtain the new iterate via $x^{(k+1)}:=x^{(k)}+\Delta x^{(k)}.$}
\end{algorithm}
One has the following convergence theorem for Algorithm \ref{algo:ssnm_GE}.
\begin{theorem}[{\cite[Theorem 5.2]{Gfrerer222}}]
\label{ThConvSSNewtonVI}
Let $\bar{x}$ be a solution of \eqref{gecp} and assume that $\partial q$ is SCD semismooth$^*$ at $(\bar{x},-F(\bar{x}))$. In addition, suppose that $F+\partial q$ is SCD regular around $(\bar{x},0)$. Then for every pair $\underline{\gamma},\bar\gamma$ with $0<\underline{\gamma}\leq \bar\gamma$ there exists a neighborhood $U$ of $\bar{x}$ such that for every starting point $x^{(0)}\in U$ Algorithm \ref{algo:ssnm_GE} produces a sequence $\{x^{(k)}\}$ converging superlinearly to $\bar{x}$, provided we choose in every iteration step $\gamma^{(k)}\in [\underline{\gamma},\bar\gamma]$.
\end{theorem}

\section{Proof of
Algorithm \ref{algo:snm_GE} as a G-semismooth Newton method}
\label{secinv}
In this section, we demonstrate that Algorithm \ref{algo:snm_GE} can be treated as an application of the G-semismooth Newton method (Algorithm \ref{alg:nm}). 

\subsection{Lipschitz continuous localization}
 
Let $F$, $G$ and $D$ be the functions and the polyhedral set in \eqref{ssnge}.
 
Suppose that \Cref{as1} holds and $\omega>0$ is the radius given by \Cref{propkey} .
Recall from \Cref{propkey} that the output $\hat u$ by running Algorithm \ref{algo:approx} can be locally represented via the solution mapping
\begin{equation}
\label{solution_map}
S(x):= \{u \ | \ 0 \in u + F(x) + \nabla G(x) N_D(G(x)+ G'(x) u) \}
\quad
\forall x\in {\mathbb B}_\omega(\bar x),
\end{equation}
which is well-defined and single-valued.
Note that $S(x)=0$ if and only if $x$ is a solution of the GE \eqref{ssnge} on ${\mathbb B}_\omega(\bar x)$.
Moreover, the corresponding multiplier calculated from Algorithm \ref{algo:approx} can also be defined as a single-valued mapping
$\Lambda(\cdot)$ on $\mathbb{B}_{\omega}(\bar x)$.
It has been established in \Cref{propkey} that $S(\cdot)$ is isolated calm at point $\bar x$ in the sense that
$\|S( x)\|\le\beta_u\| x-\bar x\|$
$\forall x \in \mathbb{B}_{\omega}(\bar x)$. 
 
In fact, a more robust result can be obtained, for which the following consequence of the reduction approach \cite[Example 3.139]{Perturbation} is necessary.

\begin{lemma} \label{Equivalence_form}
Suppose that \Cref{as1} holds and $\omega>0$ is the radius given by \Cref{propkey}.
Define for $x\in {\mathbb B}_\omega(\bar x)$ the parameterized nonsmooth equation
\[
\label{Hmapping2}
\widetilde \cH(u,\lambda; x):=
\begin{pmatrix}
u+F(x)+\nabla G(x)\lambda
\\[1mm]
G(x)+G'(x)u-\Pii_D(G(x)+G'(x)u+\lambda)
\end{pmatrix}
=0. 
\]
Then, for any $\tilde x\in{\mathbb B}_\omega(\bar x)$, 
the nondegeneracy condition holds at $(\tilde x,\tilde d ):=(\tilde x, G(\tilde x)+G'(\tilde x) S(\tilde x))$ for all $\tilde x\in\mathbb{B}_\omega(\bar x)$, where $S(\cdot)$ is defined by \eqref{solution_map}, in the sense that
\[
\label{degeneracycondition}
G'(\tilde x)\Re^n
+{\lin} T_D(\tilde d)=\Re^s\quad\forall \tilde x\in\mathbb{B}_\omega(\bar x).
\]
Moreover, let $\tilde l$ be the dimension of ${\lin} T_D(\tilde d)$ and define $\cQ:=\widetilde W^\top (T_D(\tilde d))$, where $\widetilde W\in\Re^{s\times {s-\tilde l}}$ is any matrix whose columns are linearly independent such that $\range\widetilde W={\spa}N_D(\tilde d)$. 
Then the nonsmooth equation \eqref{Hmapping2} is locally equivalent to
\[
\label{Hreduced}
\widetilde \cH_{\widetilde W}(u,\mu; x):=
\begin{pmatrix}
u+F(x)+\nabla G(x)\widetilde W\mu
\\[1mm]
\widetilde W^\top (G(x)+G'(x)u-\tilde d)
-\Pii_\cQ\big(\widetilde W^\top (G(x)+G'(x)u-\tilde d)+\mu\big)
\end{pmatrix}
=0, 
\]
in the sense that, when $x$ is sufficiently close to $\tilde x$, $(u,\lambda)$ solves \eqref{Hmapping2} if and only if $(u,\mu)$ solves \eqref{Hreduced} and $\lambda =\widetilde W \mu$. 
\end{lemma}
\begin{proof}{Proof}
Since \Cref{as1} holds, according to the proof of \citep[Proposition 5.7]{Gfrerer2021}, for any $\tilde x\in\mathbb{B}_\omega(\bar x)$ one has that $(\tilde u,\tilde \lambda):=(S(\tilde x), \Lambda(\tilde x))$ is the unique point such that \eqref{Hmapping2} holds at $(\tilde u,\tilde \lambda;\tilde x)$. 
Moreover, from \Cref{equi_form_nondegenerate} and the proof of \citep[Proposition 5.7]{Gfrerer2021} we know that the nondegeneracy condition \eqref{degeneracycondition} holds at $(\tilde x,\tilde d)$. 
From the definitions of $\tilde d$ and $\widetilde W$ we know that the mapping $d\to \widetilde W^\top (d-\tilde d)$ meets the requirements in \cite[Definition 3.135]{Perturbation}, so that $D$ is $\cC^\infty$-cone reducible to $\cQ$, which is a pointed closed convex cone.
Then, it comes from \cite[Section 4]{shapiro03} that the nonsmooth equation \eqref{Hmapping2} is locally equivalent to \eqref{Hreduced}. 
 \end{proof}

Based on the above lemma, the following result holds.
\begin{proposition}
\label{Lipschitz_of_S}
Suppose that Assumption \ref{as1} holds and $\omega>0$ is the radius given by \Cref{propkey}.
Then, both the mapping $S(\cdot)$ defined in \eqref{solution_map} and the multiplier mapping $\Lambda(\cdot)$ are Lipschitz continuous in $\mathbb{B}_{\omega}(\bar x)$.
Moreover, $S(\cdot)$ is G-semismooth in $\inti\mathbb{B}_{\omega}(\bar x)$.
\end{proposition}
\begin{proof}{Proof}
 
Let $\tilde x\in\mathbb{B}_\omega(\bar x)$, and define $(\tilde u,\tilde \lambda):=(S(\tilde x), \Lambda(\tilde x))$ and $\tilde d:=G(\tilde x)+G'(\tilde x) \tilde u$.
With $\tilde l$, $\cQ$ and $\widetilde W$ being defined the same as those in \Cref{Equivalence_form}, it comes from this lemma that the parameterized nonsmooth equation
\eqref{Hmapping2} is locally equivalent to \eqref{Hreduced} around $\tilde x$, and \eqref{degeneracycondition} holds.
In particular, $\widetilde \cH_{\widetilde W}(\tilde u,\tilde \mu; \tilde x)=0$
with $\tilde\mu \in \Re^{s-\tilde l}$ being the unique vector such that $\tilde\lambda=\widetilde W\tilde\mu$.
 
Moreover, it holds that
$$
\big\langle \Delta u,\big (( \tilde u+F( \tilde x))'_u+ (\nabla G( \tilde x) \widetilde W\tilde \mu)'_u\big)
\Delta u
\big\rangle =\|\Delta u\|^2.
$$
Therefore, by following the proof of \cite[Proposition 2]{meng} and using \cite[Corollary 2]{meng} we know that there exists a neighborhood $\mathbb{O}(\tilde x)$ of $\tilde x$ such that $S(\cdot )$ and $\Lambda(\cdot)$ are Lipschitz continuous.
Since $\mathbb{B}_{\omega}(\bar x)$ is a compact set, for all $\tilde x\in \mathbb B_\omega(\bar x)$, one can always find a finite collection $\chi$ of $\tilde x$ such that the union of these open neighborhoods $\cup_{\tilde x\in \chi} \mathbb{O}(\tilde x)$ covers $\mathbb{B}_{\omega}(\bar x)$. Therefore, $S(\cdot)$ and $\Lambda(\cdot)$ are Lipschitz continuous in $\mathbb{B}_\omega(\bar x)$.
Finally, as the projection onto $\cQ$ is strongly G-semismooth, by \cite[Corollary 2]{meng}
we know that $S(\cdot)$ is G-semismooth with respect to $\partial_B S$ or $\partial S$. This completes the proof.
 \end{proof}

According to Rademacher's theorem, the solution mapping $S(\cdot)$ defined in \eqref{solution_map} is almost everywhere differentiable in $\mathbb{B}_{\omega}(\bar x)$ due to \Cref{Lipschitz_of_S}. It is not hard to compute the Frech\'et derivative of $S$ when it is differentiable, but
it is hard to directly compute the corresponding Bouligand subdifferential \eqref{bouligand} or Clarke's generalized Jacobian by taking limits, since $S(\cdot)$ is implicitly defined.
Therefore, to implement Algorithm \ref{alg:nm} to solve the nonsmooth equation $S(x)=0$, the corresponding mapping $\cT$ should be explicitly computed, which will be done in the next part.

\subsection{G-Semismoothness}
The analysis in \cite{meng} for locally Lipschitz continuous homeomorphisms can be utilized to compute a set-valued mapping such that the solution mapping $S(\cdot)$ defined in \eqref{solution_map} is G-semismooth with respect to it.
Specifically, we have the following key result which gives the G-semismoothness of $S$ around $\bar x$.

\begin{proposition}
\label{proppartialb}
Suppose that Assumption \ref{as1} holds and
$\omega>0$ is the radius given by \Cref{propkey}.
Then the solution mapping $S(\cdot)$ defined in \eqref{solution_map} is G-semismooth at every $\tilde x\in\inti \mathbb{B}_{\omega}(\bar x)$ with respect to the set-valued mapping
\small
\begin{equation}
\label{partialbs}
\begin{array}{ll}
\cT_S(x):=
\left\{
\begin{array}{r}
-\cL_{ \lambda}'( x)
+\nabla G( x)W
\big [{W}^\top G'({x})\nabla G( x) W \big]^{-1} W ^\top (G'( x)
\cL_{ \lambda}'( x)
-[G( x)+G'( x) u]'_x)
\\[2mm]
\mid W \mbox{ has full column rank}, \,
\range W=({\spa}\cF)^\perp
\mbox{ with }
\cF \mbox{ being a face of } \cK_D(d,\lambda)
\end{array}
\right\},
\end{array}
\end{equation}
\normalsize
where $u:=S(x)$, $d := G( x)+G'( x)S( x)$ and $\lambda:=\Lambda(x)$. Moreover, $\cT_S(\cdot)$ is outer semicontinuous at every $\tilde x\in \inti \mathbb{B}_{\omega}(\bar x)$.
\end{proposition}

\begin{proof}{Proof}
 
Let $\tilde x\in \inti \mathbb{B}_{\omega}(\bar x)$ and define $\tilde u :=S(\tilde x)$,
$\tilde d:=G(\tilde x)+G'(\tilde x)\tilde u$, and
$\tilde\lambda:= \Lambda(\tilde x)$. 
Let $\tilde l$, $\widetilde W$ and $\cQ$ be the same as those defined in \Cref{Equivalence_form}. 
We know from \Cref{Equivalence_form} that the parameterized nonsmooth equation \eqref{Hmapping2} is locally equivalent to \eqref{Hreduced} around $\tilde x$, and \eqref{degeneracycondition} holds. 
Then by following the analysis of \cite[Proposition 2]{meng} the mapping
\begin{equation*}
\Psi (v,\varsigma,x):=
\begin{pmatrix}
v+F(x)+\nabla G(x)\widetilde W\varsigma
\\[.5mm]
\widetilde W^\top[G(x)+G'(x)v-\tilde d]-\Pii_Q(\widetilde W^\top[G(x)+G'(x)v-\tilde d]+\varsigma)
\\[.5mm]
x
\end{pmatrix}
\end{equation*}
is locally Lipschitz homeomorphism around $(\tilde u,\tilde\mu,\tilde x)$, where $\tilde \mu$ is the unique vector such that $\tilde\lambda = \widetilde W \tilde \mu$.
Note that
\begin{equation*}
\big(\widetilde W^\top[G(x)+G'(x)v-\tilde d]+\varsigma\big)'(v,\varsigma,x)=
\big(\widetilde W^\top G'(x),
I,
\widetilde W^\top [G(x)+G'(x)v]'_x\big),
\end{equation*}
which is always surjective.
Consequently, from \cite[Lemma 2.1]{sunmor} we know that
\[
\label{partialpsi}
\partial_B \Psi( v,\varsigma, x)=\left\{
\begin{array}{r}
\begin{pmatrix}
I & \nabla G( x)\widetilde W & \cL_{\widetilde W\varsigma}'(x)
\\
(I-\Xi)\widetilde W^\top G'( x)
&
-\Xi
&
(I-\Xi)\widetilde W^\top [G(x)+G'(x) v]'_x
&
\\[1mm]
0&0&I
\end{pmatrix}
\qquad
\\[2mm]
\mid \Xi\in \partial_B \Pii_\cQ (\widetilde W^\top [G(x)+G'(x)v-\tilde d]+\varsigma)
\end{array}
\right\}.
\]
Since the projection operator $\Pii_\cQ$ is (strongly) semismooth, one has $\Psi$ is also semismooth, and it follows from \cite[Theorem 2]{meng} that $\Psi^{-1}$ is semismooth at $(0,0, \tilde x)\in\Re^{n}\times\Re^{s-\tilde l}\times \Re^n$. Moreover, from \cite[Lemma 2]{meng} we know that all the elements of $\partial_B \Psi( v,\varsigma, x)$ are nonsingular whenever $(v,\varsigma,x)$ is sufficiently close to $(\tilde u,\tilde \mu,
\tilde x)$.
It can be observed from \eqref{partialpsi} that each element of $\partial_B\Psi$ is nonsingular at $(v,\varsigma,x)$ if and only if each matrix
$\Xi+(I-\Xi)\widetilde W^\top G'( x)\nabla G( x)\widetilde W$ is nonsingular for all $\Xi\in \partial_B \Pii_\cQ(\widetilde W^\top [G(x)+G'(x)v-\tilde d]+\varsigma)$.
Therefore, for all $(v,\varsigma,x)$ sufficiently close to
$(\tilde u,\tilde \mu,\tilde x)$,
it holds by elementary column transformations that
\small
$$
\begin{array}{ll}
\partial_B \Psi( v,\varsigma, x)
=
\left\{
\begin{array}{l}
\begin{pmatrix}
& I-\nabla G(x)\widetilde W (\Gamma(\Xi))^{-1}
(I-\Xi)\widetilde W^\top G'( x)
&\nabla G(x)\widetilde W (\Gamma(\Xi))^{-1}
& \cE_S (v,\varsigma,\Xi)
\\[2mm]
&
(\Gamma(\Xi))^{-1}
(I-\Xi)\widetilde W^\top G'(x)
&-(\Gamma(\Xi))^{-1}&
 \cE_\Lambda (v,\varsigma,\Xi)
\\[2mm]
&0&0&I
\end{pmatrix}^{-1} \quad
\\[4mm]
\ \big | \
\Gamma(\Xi):=\Xi+(I-\Xi)\widetilde W^\top G'( x)\nabla G( x)\widetilde W
\ \mbox{with}\
\Xi\in \partial_B \Pii_\cQ (\widetilde W^\top [G(x)+G'(x)u-\tilde d]+\varsigma)
\end{array}
\right\},
\end{array}
$$
\normalsize
where
\small
\begin{equation*}
\left\{
\begin{array}{ll}
 \cE_S (v,\varsigma,\Xi):=
-\cL_{\widetilde{W}\varsigma}'(x)
+\nabla G(x)\widetilde W(\Gamma(\Xi))^{-1}(I-\Xi)\widetilde W^\top (G'( x)
\cL_{\widetilde W\varsigma}'(x)
-[G(x)+G'(x)v]'_x),
\\[2mm]
 \cE_\Lambda (v,\varsigma,\Xi):
=-(\Gamma(\Xi))^{-1}(I-\Xi)\widetilde W^\top (G'(x)\cL_{\widetilde{W}\varsigma}'(x)
-[G(x)+G'(x)v]'_x).
\end{array}
\right.
\end{equation*}
\normalsize
Note that for all $x$ sufficiently close to $\tilde x$, it holds that
$(u;\mu;x) =\Psi^{-1} (0,0, x)$ where $\mu$ is the unique vector such that $\lambda=\widetilde W\mu$.
Since $\Psi^{-1}$ is G-semismooth,
the solution mapping $S(\cdot)$ is G-semismooth at $\tilde x$ with respect to
\small
\begin{equation*}
 \widetilde \cE_S (x):=
\left\{
\begin{array}{ll}
-\cL_{\lambda}'(x)
+\nabla G(x)\widetilde W(\Gamma(\Xi))^{-1}(I-\Xi)\widetilde W^\top (G'( x)
\cL_{\lambda}'(x)
-[G(x)+G'(x)u]'_x)
\quad
\\[2mm]
\quad\mid 
\Gamma(\Xi):=\Xi+(I-\Xi)\widetilde W^\top G'( x)\nabla G( x)\widetilde W
\mbox{ with }
\Xi\in \partial_B \Pii_\cQ (\widetilde W^\top [G(x)+G'(x)u-\tilde d]+\mu)
\end{array}
\right\}.
\end{equation*}
\normalsize
Moreover, $ \widetilde \cE_S (\cdot)$ is outer semicontinuous around $\tilde x$.
According to \Cref{partialbpiq}, one has for all $x$ sufficiently close to $\tilde x$,
\[
\label{partialbstemp}
\begin{array}{ll}
 \widetilde \cE_S (x)
=\left\{
\begin{array}{l}
-\cL_{\lambda}'( x)
+\nabla G( x)\widetilde W(\tilde\Gamma_\cG 
)^{-1}
\Pii_{(\spa \cG)^\perp}
\widetilde W^\top (G'( x)
\cL_{ \lambda}'( x)
-[G(x)+G'( x) u]'_x)
\\[2mm]
\ \mid 
\begin{array}{l}
\tilde \Gamma_\cG
:=\Pii_{\spa \cG}+ \Pii_{({\spa} \cG)^\perp}
\widetilde W^\top G'( x)\nabla G( x)\widetilde W, \ 
\cG \mbox{ is a face of } \cK_Q( \widetilde W^\top [d-\tilde d],\mu)
\end{array}
\end{array}
\right\}.
\end{array}
\]
Let $\cG$ be a face of $\cK_Q( \widetilde W^\top [d-\tilde d],\mu)$.
One can find two matrices $U_1$ and $U_2$, whose columns form an orthonormal basis of ${\spa}\cG$ and $({\spa}\cG)^\perp$, respectively.
Note that $\Pii_{{\spa}\cG}= U_1U_1^\top $ and $\Pii_{({\spa}\cG)^\perp}= U_2U_2^\top$.
Consequently,
\begin{equation*}
\begin{array}{ll}
(\tilde \Gamma_\cG)^{-1}\Pii_{(\spa \cG)^\perp}
=
[U_1U_1^\top
+
U_2U_2^\top
\widetilde W^\top G'( x)\nabla G( x)\widetilde W]^{-1}
U_2U_2^\top .
\end{array}
\end{equation*}
For any $w $ and $\nu$ such that $\nu =[U_1U_1^\top + U_2U_2^\top \widetilde W^\top G'( x)\nabla G(x)\widetilde W]^{-1}
U_2U_2^\top w $, it holds that
\begin{equation*}
\begin{array}{ll}
&
U_1U_1^\top \nu + U_2U_2^\top \widetilde W^\top G'( x)\nabla G(x)\widetilde W \nu =
U_2U_2^\top w
\\[1mm]
\Rightarrow&
U_1^\top \nu = 0
\ 
\mbox{and}
\ 
U_2^\top \widetilde W^\top G'(x)\nabla G(x)\widetilde W \nu =
U_2^\top w
\\[1mm]
\Rightarrow&
 \nu=U_2\xi \mbox{ for some } \xi, \mbox{ and } U_2^\top \widetilde W^\top G'(x)\nabla G(x)\widetilde W U_2\xi =
U_2^\top w .
\end{array}
\end{equation*}
Therefore, one gets $\nu=
U_2\xi
=U_2 [U_2^\top \widetilde W^\top G'( x)\nabla G(x)\widetilde W U_2]^{-1}
U_2^\top w $, so that
\[
\label{invimportant}
\widetilde W
(\tilde \Gamma_\cG)^{-1}\Pii_{(\spa \cG)^\perp}\widetilde W^\top
=
\widetilde WU_2 [U_2^\top \widetilde W^\top G'( x)\nabla G(x)\widetilde W U_2]^{-1}U_2^\top \widetilde W^\top .
\]
Recall that $\lambda=\widetilde W \mu $, so that $[\lambda]^\perp=\{e\mid \langle \widetilde W^\top e,\mu \rangle=0\}$,
which implies $\widetilde W^\top ( [\lambda]^\perp)=[\mu ]^\perp$.
One has from \cite[Theorem 6.31]{rock1998} that
\begin{equation}
\label{kdexpress}
\cK_D( d,\lambda)
=\{e\mid e\in T_D(d), \langle e,\lambda\rangle=0\}
=\{e\mid\widetilde W^\top e\in T_\cQ(\widetilde W^\top [d-\tilde d]), \langle \widetilde W^\top e,\mu \rangle=0\}.
\end{equation}
Therefore, it holds that
\[
\label{criticalconeeq}
\widetilde W^\top \cK_D( d,\lambda)
= \cK_Q( \widetilde W^\top [d-\tilde d],\mu).
\]
Consequently, one has for any given $\upsilon \in\Re^{s-\tilde l}$,
$$
\begin{array}{ll}
\sup\limits_{e\in \cK_D( d,\lambda)}\{\langle \widetilde W \upsilon, e\rangle \}
=
\sup\limits_{e\in \cK_D( d,\lambda)}\{\langle \upsilon, \widetilde W^\top e\rangle \mid e\in \cK_D(d,\lambda)\}
=
\sup\limits_{\nu \in\Re^{s-\tilde l}} \{\langle \upsilon, \nu\rangle\mid \nu\in \cK_Q( \widetilde W^\top [d-\tilde d],\mu)\}.
\end{array}
$$
Therefore, $\upsilon\in \cK_Q( \widetilde W^\top [d-\tilde d],\mu)^\circ$ if and only if $\widetilde W\upsilon\in \cK_D( d,\lambda)^\circ$.
Recall that $\cG$ is a face of $\cK_Q( \widetilde W^\top [d-\tilde d],\mu)$, i.e., there exists a nonzero vector
$\tilde \upsilon \in \Re^{s-\tilde l}$ in its polar such that $\cG=\cK_Q( \widetilde W^\top [d-\tilde d],\mu)\cap [\tilde \upsilon]^{\perp}$.
Note that
\begin{equation*}
\begin{array}{ll}
\tilde\cF:=
\{ p\mid \widetilde W^\top p \in \cG\}
=
\{ p\mid \widetilde W^\top p \in \cK_Q( \widetilde W^\top [d-\tilde d],\mu)\cap [\tilde \upsilon]^{\perp} \}=
\cK_D( d,\lambda)\cap [\widetilde W\tilde \upsilon]^{\perp}.
\end{array}
\end{equation*}
Since $\widetilde W\tilde \upsilon \in \cK_D(d,\lambda)^\circ$,
the set $\tilde\cF$ is a face of $\cK_D( d,\lambda)$.
Moreover, one has
${\spa}\tilde\cF=\{p\mid \widetilde W^\top p\in{\spa}\cG\}$.
Recall that $\range U_1={\spa}\cG$, so that
$\spa \tilde\cF=\{p\mid \widetilde W^\top p\in {\range } U_1\}
=\{p\mid U_2^\top \widetilde W^\top p=0\}$.
Therefore, if $W$ is a matrix with full column rank such that
${\range }W=(\spa\tilde\cF)^\perp=\range (\widetilde W U_2)$,
there exists a nonsingular square matrix $P$ such that
$W=\widetilde W U_2 P$.
In this case, one has
$$
\begin{array}{ll}
W
\big [{W}^\top G'({x})\nabla G( x) W \big]^{-1} W ^\top
=\widetilde W U_2 P
\big [P^\top (\widetilde W U_2)^\top G'({x})\nabla G( x)\widetilde W U_2 P \big]^{-1} P^\top (\widetilde W U_2)^\top
\\[2mm]
=\widetilde W U_2 P
\big [(\widetilde W U_2)^\top G'({x})\nabla G( x) \widetilde W U_2 P \big]^{-1} (P^\top )^{-1}P^\top (\widetilde W U_2)^\top
\\[2mm]
=\widetilde W U_2
\big [(\widetilde W U_2)^\top G'({x})\nabla G( x)\widetilde W U_2 \big]^{-1} (\widetilde W U_2)^\top .
\end{array}
$$
Such an equality, together with \eqref{partialbs}, \eqref{partialbstemp}, and \eqref{invimportant}, implies that
\[
\label{partialbinclusionfinal}
 \widetilde \cE_S (x)\subseteq
\cT_S(x).
\]
Next, we show that the inclusion in \eqref{partialbinclusionfinal} is an equality.
Let $\cF$ be an arbitrary face of $\cK_D( d, \lambda)$, i.e., there exists
$
\breve \lambda\in \cK_D( d, \lambda)^\circ
=N_D( d)+{\spa}[\lambda]\subseteq{\spa}N_D( d)\subseteq{\spa}N_D( \tilde d)$
such that
$\cF=\cK_D( d, \lambda)\cap[\breve \lambda]^\perp$.
Since $\breve\lambda\in\range\widetilde W$,
one has $\breve\lambda=\Pii_{\range\widetilde W}(\breve\lambda)=\widetilde W[\widetilde W^\top \widetilde W]^{-1}\widetilde W^\top\breve \lambda$.
Then by \eqref{kdexpress} one can get
\begin{equation*}
\cK_D( d,\lambda)\cap[\breve \lambda]^\perp
=\big\{
e\mid\widetilde W^\top e\in T_\cQ(\widetilde W^\top(d-\tilde d)), \langle \widetilde W^\top e,\mu \rangle=0,
\langle \widetilde W^\top e, [\widetilde W^\top \widetilde W]^{-1}\widetilde W^\top\breve \lambda\rangle=0
\big\}.
\end{equation*}
Therefore, it holds that $
\widetilde W^\top \cF
=\cK_Q( \widetilde W^\top [d-\tilde d],\mu)\cap [(\widetilde W^\top \widetilde W)^{-1}\widetilde W^\top \breve \lambda]^\perp$.
Then, from \eqref{criticalconeeq} one has
$
[\widetilde W^\top \widetilde W]^{-1}\widetilde W^\top \breve\lambda
\in (\widetilde W^\top \cK_D( d,\lambda))^\circ
=(\cK_Q( \widetilde W^\top [d-\tilde d],\mu))^\circ$.
Therefore, it holds that
\begin{equation*}
\langle [\widetilde W^\top \widetilde W]^{-1}\widetilde W^\top \breve\lambda, \widetilde W^\top \hat d\rangle
=
\langle \widetilde W [\widetilde W^\top \widetilde W]^{-1}\widetilde W^\top \breve\lambda, \hat d\rangle
=
\langle \breve\lambda, \hat d\rangle\le 0
\quad
\forall\hat d\in \cK_D( d,\lambda).
\end{equation*}
Consequently, $\widetilde W^\top \cF$ is exactly a face of $\cK_Q( \widetilde W^\top [d-\tilde d],\mu)$.
Thus, the inclusion in \eqref{partialbinclusionfinal} holds as an equality, i.e., $ \cT_S( x)= \widetilde \cE_S (x)$ for all $x$ sufficiently close to $\tilde x$.
Therefore, $\cT_S(\cdot)$ is also outer semicontinuous at any $\tilde x\in \inti \mathbb{B}_{\omega}(\bar x)$. This completes the proof.
 \end{proof}

\subsection{Regularity conditions}
The following result is crucial for using Algorithm \ref{alg:nm} to solve $S(x)=0$.

\begin{proposition}
\label{propbd}
Under \Cref{as1}, it holds that \Cref{ass_2} is equivalent to the regularity condition that every element of $\cT_S(\bar x)$ defined in \eqref{partialbs} is nonsingular.
\end{proposition}
\begin{proof}{Proof}
Since \Cref{as1} holds, one has $\bar u=S(\bar x)$.
Then by \Cref{proppartialb} it holds that
\begin{equation*}
\begin{array}{ll}
\cT_S(\bar x)
=
\left\{
\begin{array}{l}
-\cL_{\bar \lambda}'(\bar x)
+\nabla G(\bar x)W
\big [{W}^\top G'(\bar{x})\nabla G(\bar x) W \big]^{-1} W ^\top (G'(\bar x)
\cL_{\bar \lambda}'(\bar x)
-G'(\bar x))
\\[1mm]
\equiv
-(I-\Pii_{\range(\nabla G(\bar x)W)})\cL_{\bar\lambda}'(\bar x)
-\Pii_{\range(\nabla G(\bar x)W)}
\\[1mm]
\mid W \mbox{ has full column rank, }
\range W=({\spa}\cF)^\perp
\mbox{ with }
\cF \mbox{ being a face of } \cK_D(\bar d,\bar\lambda)
\end{array}
\right\},
\end{array}
\end{equation*}
where $\bar d := G(\bar x)+G'(\bar x)S(\bar x)$ and $\bar \lambda := \Lambda(\bar x)$.
Let $\cF$ be an arbitrary face of $\cK_D(G(\bar x),\bar \lambda)$
with $W$ having full column rank such that $\range W=({\spa}\cF)^\perp$.
Let $Z$ be an arbitrary matrix (with full column rank) such that
$\range Z=\{u\mid G'(\bar x )u\in {\spa} \cF\}={\ker}(W^\top G'(\bar x))$.
Note that ${\range(\nabla G(\bar x)W)}={\ker}Z^\top$.
Therefore, it holds that
\begin{equation*}
\begin{array}{lll}
&-(I-\Pii_{\range(\nabla G(\bar x)W)})\cL_{\bar\lambda}'(\bar x)\nu
-\Pii_{\range(\nabla G(\bar x)W)} \nu\neq0\quad\forall \nu\neq 0
\\[1mm]
\Leftrightarrow &
(I-\Pii_{\range(\nabla G(\bar x)W)})\cL_{\bar\lambda}'(\bar x)\nu\neq 0
\quad\forall \nu\neq 0 \mbox{ such that }
\Pii_{\range(\nabla G(\bar x)W)} \nu= 0
\\[1mm]
\Leftrightarrow &
Z (Z^\top Z)^{-1}Z^\top \cL_{\bar\lambda}'(\bar x)\nu\neq 0
\quad\forall 0\neq \nu \in \range Z
\\[1mm]
\Leftrightarrow &
Z^\top \cL_{\bar\lambda}'(\bar x)\nu\neq 0
\quad\forall 0\neq \nu \in \range Z.
\end{array}
\end{equation*}
Thus, \Cref{ass_2} is equivalent to the condition that every element of $\cT_S(\bar x)$ is nonsingular.
This completes the proof.
 \end{proof}

\subsection{Equivalence to a G-Semismooth Newton method}
We are ready to show that the semismooth* Newton method in algorithm \ref{algo:snm_GE} is exactly a special case of the G-semismooth Newton method in Algorithm \ref{alg:nm}.
Suppose that Assumptions \ref{as1} and \ref{ass_2} hold.
Let $\omega$ be the parameter defined in \Cref{propkey}.
For any $\tilde x \in \mathbb{B}_\omega(\bar x)$ with
$\tilde d=G(\tilde x)+G'(\tilde x)S(\tilde x)$ and $\tilde u=S(\tilde x)$,
one has $\tilde\lambda\in N_D(\tilde d)$.
Moreover, it is easy to see that
$\bar\cF:={\lin}T_D(\tilde d)={\lin}\cK_D(\tilde d,\tilde \lambda)$
is a face of $\cK_D(\tilde d,\tilde \lambda)$.
In fact, from the proof of \Cref{proppartialb} one can see that
$\bar\cF$ is exactly the face such that $\widetilde W^\top \bar\cF=\cG\equiv\{0\}$, where $\widetilde W$ is a matrix with full column rank such that $
\range\widetilde W={\spa}N_D(\tilde d)$, while $\cG$ is a face of $\cK_\cQ(0, \tilde\mu)=\widetilde W^\top \cK_D(\tilde d,\tilde \lambda)$.
Consequently, the columns of $\widetilde W$ form a basis of
$({\spa}\bar\cF)^{\perp}$, so that by \eqref{partialbs} one can get
\[
\label{mx}
\begin{array}{ll}
V_{\tilde x}:&=
-\cL_{\tilde \lambda}'(\tilde x)
+\nabla G(\tilde x)\widetilde W
\big [{\widetilde W}^\top G'(\tilde {x})\nabla G(\tilde x) \widetilde W \big]^{-1} \widetilde W^ \top (G'(\tilde x)
\cL_{\tilde \lambda}'(\tilde x)
-[G(\tilde x)+G'(\tilde x)\tilde u]'_x)
\\[1mm]
&=-\Pii_{\cZ} \mathcal{L}'_{\tilde \lambda}(\tilde x)
- \Pii_{\cW} - \nabla G(\tilde x)\widetilde W
\big [\widetilde {W}^\top G'(\tilde {x})\nabla G(\tilde x)\widetilde W \big]^{-1}
\widetilde W ^\top (G'(\tilde x) \tilde u)_x'
\in\cT_S(\tilde x),
\end{array}
\]
where $\cW:=\range(\nabla G(\tilde x)\widetilde W)$ and $\cZ:=\cW^{\perp}$.
Based on the results established in the previous two subsections, we can apply Algorithm \ref{alg:nm} to solve problem \eqref{ssnge}.
The resulting implementation is given as follows.

\begin{algorithm}[H]
\caption{A G-Semismooth Newton method for solving \eqref{ssnge} \label{algo:snm_GEx}}

\Input {$F:\Re^n\to\Re^n$, $G:\Re^n\to\Re^s$, $D \subset \Re^s$,
$x^{(0)}\in\Re^n$, and $\varrho\ge 0.$}
\Output {$\{x^{(k)}\}.$}
\For{$k=0,1, \ldots$,}{
1. if $x^{(k)}$ solves \eqref{ssnge}, stop the algorithm;
\\[1mm]
2. run the approximation step in Algorithm \ref{algo:approx} with input $x^{(k)}$ to compute $\hat u^{(k)}$, $\hat{\lambda}^{(k)}$, $\hat{d}^{(k)}$ and $\mathcal{L}_{\hat{\lambda}^{(k)}}(x^{(k)})$;
\\[1mm]
3. compute $V^{(k)}$ such that ${\dist}(V^{(k)},\cT_S(x^{(k)}))\le \varrho \|{\hat u^{(k)}}\|$ with $\cT_S$ being given in \eqref{partialbs};
\\[1mm]
4. compute the Newton direction $\Delta x^{(k)}$ satisfying
$V^{(k)}\Delta x^{(k)}+{\hat u^{(k)}}=0$,
and set $x^{(k+1)}:=x^{(k)}+\Delta x^{(k)}$.
}
\end{algorithm}

The following result shows that
Algorithm \ref{algo:snm_GE} is a special case of Algorithm
\ref{algo:snm_GEx}.

\begin{theorem}
Suppose that Assumptions \ref{as1} and \ref{ass_2} hold.
Then, Algorithm \ref{algo:snm_GE} is an instance of
Algorithm \ref{alg:nm} (in the form of Algorithm \ref{algo:snm_GEx}) in the sense that the local superlinear convergence of Algorithm \ref{algo:snm_GE} can be obtained from \Cref{thmgss}.
\end{theorem}

\begin{proof}{Proof}
Let $\{ x^{(k)}\}$ be the sequence generated by Algorithm \ref{algo:snm_GE}.
For a fixed $\bar k\ge 1$ such that $x^{(\bar k)}$ is well-defined, we assume that $x^{(\bar k-1)}$ is also the output of the final step at the iteration indexed by $(\bar k-1)$ of Algorithm \ref{algo:snm_GEx}.
Then, the first two steps at the iteration indexed by $\bar k$ of Algorithm \ref{algo:snm_GEx} are the same as the first two steps at the iteration indexed by $\bar k$ of Algorithm \ref{algo:snm_GE}.
During the third step of Algorithm \ref{algo:snm_GEx} at iteration $\bar k$,
one can take
\begin{equation*}\cW^{(\bar k)}:=\range(\nabla G( x^{(\bar k)})\hat W^{(\bar k)}),
\quad
\cZ^{(\bar k)}:=(\cW^{(\bar k)})^{\perp}
\quad
\mbox{and}
\quad
V^{(\bar k)}:=
-\Pii_{\cZ^{(\bar k)}} \mathcal{L}'_{\hat \lambda^{(\bar k)}}( x^{(\bar k)})
- \Pii_{\cW^{(\bar k)}},
\end{equation*}
where $\hat W^{(\bar k)}$ comes from the third step of Algorithm
\ref{algo:snm_GE} at the iteration indexed by $\bar k$.
Then, from \eqref{mx} one has
\small
\begin{equation*}
\dist\big(V^{(\bar k)}, \cT_S( x^{(\bar k)})\big)
\le
\|\nabla G( x^{(\bar k)})\hat W^{(\bar k)}
\big [(\hat W^{(\bar k)})^\top G'( {x}^{(\bar k)})\nabla G( x^{(\bar k)})\hat W^{(\bar k)} \big]^{-1}
(\hat W^{(\bar k)}) ^\top (G'( x^{(\bar k)}) \hat u^{(\bar k)})_x'\|,
\end{equation*}
\normalsize
where $ \hat u^{(\bar k)}:=S(x^{(\bar k)})$.
Then, the corresponding $\Delta x^{(\bar k)}$ computed by the final step at the iteration indexed by $\bar k$ of Algorithm \ref{algo:snm_GEx} is the same as the one calculated by \eqref{iteration_scheme_snm_GE} at the iteration indexed by $\bar k$.
Let $\omega$ be the parameter specified in \Cref{propkey}. Recall that for any $\tilde x \in \mathbb{B}_{\omega}(\bar x)$, $V_{\tilde x}$ in \eqref{mx} is independent of the specific choice of the corresponding $\widetilde W$ in \eqref{mx}.
Without loss of generality, one can assume that
$\widetilde W$ in \eqref{mx} is uniformly bounded.
One has $[\widetilde W^\top G'(\tilde x)\nabla G(\tilde x)\widetilde W]^{-1}$ is well-defined and uniformly bounded since the nondegeneracy condition holds and $G$ is continuously differentiable.
Furthermore, since $G'(\cdot)$ is also continuously differentiable, one can get
${\dist}(V^{(\bar k)},\cT_S(x^{(\bar k)}))\le \varrho \|{\hat u^{(\bar k)}}\|$
with
\begin{equation*}
\varrho:=\sup_{\tilde x\in \mathbb{B}_{\omega}(\bar x)}
\left \{
\|\nabla G(\tilde x) \widetilde W
\big [\widetilde W^\top G'(\tilde x)\nabla G(\tilde x)\widetilde W\big]^{-1}
 \widetilde W ^\top G''(\tilde x)\|
\right \}<\infty.
\end{equation*}
Consequently, the iteration sequence $\{ x^{(k)}\}$ generated by Algorithm \ref{algo:snm_GE} can be viewed as the one generated by 
 Algorithm \ref{algo:snm_GEx}. 
 
Hence, if $x^{(0)}$ is sufficiently close to $\bar x$, 
 one has from \Cref{proppartialb} and \Cref{propbd} that the local superlinear convergence of Algorithm \ref{algo:snm_GE} is guaranteed by \Cref{thmgss}.
This completes the proof.
 \end{proof}

\section{Proof of Algorithm \ref{algo:ssnm_GE} as a G-semismooth Newton method}
\label{secscd}
In this section, we show that Algorithm \ref{algo:ssnm_GE} is
also an application of Algorithm \ref{alg:nm}. 
 The methodology developed here also can be used to show that the implementable SCD semismooth* Newton method in the more recent work \cite{gf23} is also a G-semismooth Newton method.
 
We first provide the following two preliminary results.
\begin{lemma}
\label{lem:spq}
Let $q$ be the function in \eqref{gecp} and $\cP_{{\gamma}^{-1}q} $ be the proximal mapping defined by \eqref{proxdef} with $\gamma>0$.
For any $y,z\in\Re^n$ such that $(z;y)\in\gph \cP_{{\gamma}^{-1}q} $, it holds that
\begin{equation}
\label{spq}
\begin{array}{ll}
\{\mathrm{rge}(I;B)\mid B\in\partial_{B}\mathcal{P}_{\gamma^{-1}q}(z)\}
&=
\cS^*_{\cP_{{\gamma}^{-1}q} } (z,y)\\
&=
 
\left\{\left\{(-(e^*+\gamma e);
-\gamma e)
\mid
(e ;e^*)\in
L\right\}
\mid
L \in \cS^*_{\partial q}(y,\gamma(z-y))\right\},
\end{array}
\end{equation}
 
where the definitions of $\cS^*_{\cP_{{\gamma}^{-1}q} }$ and $\cS^*_{\partial q}$ come from \Cref{DefSCDProperty}.
\end{lemma}
\begin{proof}{Proof}
Since $q$ is a closed proper convex function, one has for any $y,z\in\Re^n$,
\[
\label{incluchange}
(z;y)\in\gph \cP_{{\gamma}^{-1}q}
\quad
\Leftrightarrow
\quad
\gamma (z-y)\in\partial q(y)
\quad
\Leftrightarrow
\quad
\phi
(z,y):=
(y;\gamma (z-y))\in\gph\partial q.
\]
Meanwhile, one has $\phi'(z,y)=
\begin{pmatrix}
0& I
\\
\gamma I& -\gamma I
\end{pmatrix}$, which is nonsingular for all $\gamma>0$.
Note that
\begin{equation*}
\begin{pmatrix}
0&-I
\\
I& 0
\end{pmatrix}
\begin{pmatrix}
0& \gamma I
\\
I& -\gamma I
\end{pmatrix}
\begin{pmatrix}
0&-I
\\
I& 0
\end{pmatrix}^{ \top}
\begin{pmatrix}
e
\\
e^*
\end{pmatrix}
=
\begin{pmatrix}
-(e^*+\gamma e)
\\
-\gamma e
\end{pmatrix}.
\end{equation*}
Thus from \cite[ Lemma 3.11 and Proposition 3.14]{gfrerer2022local} and 
\cite[Theorem 13.52]{rock1998} one gets \eqref{spq}.
 \end{proof}

\begin{lemma}
\label{prop:uineq}
Let $u_{\gamma}$ be the function defined in \eqref{proximal} with $\gamma>0$.
Suppose that $\bar x$ is a solution to the GE \eqref{gecp}.
For any $ x\in\Re^n$, by setting
$ u:=u_\gamma( x)$, $ z:= x-{\gamma}^{-1}F( x)$ and
$\bar z:=\bar x-{\gamma}^{-1}F(\bar x)$, one has
\begin{equation*}
\| u+ x-\bar x\|^2
+\|F(\bar x)-F( x)-\gamma u\|^2
\le
\max\{1,\gamma^2\}
\| z-\bar z \|^2.
\end{equation*}
\end{lemma}
\begin{proof}{Proof}
Note that $ u+ x=\cP_{{\gamma}^{-1}q} ( z)$.
One has 
$(u+ x; -F( x)-\gamma u)
=
\big( u+ x; \gamma( z-( u+ x))\big)\in\gph \partial q$ from 
\eqref{incluchange}.
Moreover, one has
$\bar x=\cP_{{\gamma}^{-1}q} (\bar z)$, so that
\begin{equation*}
\begin{array}{ll}
\| u+ x-\bar x\|^2
+\| F(\bar x) -F( x)-\gamma u\|^2
\\[1mm]
=
\|\cP_{{\gamma}^{-1}q} ( z)
-\cP_{{\gamma}^{-1}q} (\bar z)\|^2
+\| F(\bar x)-F( x)-\gamma (\cP_{{\gamma}^{-1}q} ( z)- x)\|^2
\\[1mm]
=
\|\cP_{{\gamma}^{-1}q} ( z)
-\cP_{{\gamma}^{-1}q} (\bar z)\|^2
+\| F(\bar x)-\gamma \bar x+\gamma x -F( x)-\gamma (\cP_{{\gamma}^{-1}q} ( z)-\bar x)\|^2
\\[1mm]
=
\|\cP_{{\gamma}^{-1}q} ( z)
-\cP_{{\gamma}^{-1}q} (\bar z)\|^2
+\gamma^2\| z-\bar z
-(\cP_{{\gamma}^{-1}q} ( z)-\cP_{{\gamma}^{-1}q} (\bar z))
\|^2
\le
\max\{1,\gamma^2\}
\| z-\bar z \|^2,
\end{array}
\end{equation*}
where the last inequality comes from \cite[Proposition 1(c)]{rtrmonotone}. This completes the proof.
 \end{proof}

\subsection{Lipschitz continuous localization and G-semismoothness}

Recall that the GE in \eqref{gecp} is equivalent to the nonsmooth equation $u_{\gamma}(x)=0$ (for any $\gamma>0$) with $u_{\gamma}$ in \eqref{proximal}. 
Let $U$ be a neighborhood of $\bar x$ such that $F$ is Lipschitz continuous on it with modulus $\ell>0$.
According to \cite[eq. (5.13)]{Gfrerer222} one has
$$
\begin{array}{l}
\|u_{\gamma} (x)-u_\gamma (x')\|\le 2\|x-x'\|+{\gamma}^{-1}\|F(x)-F(x')\|
\le \big(2+\frac{\ell}{\gamma}\big)\|x-x'\|\quad\forall x,x'\in U.
\end{array}
$$
Therefore,
for any $\gamma>0$,
the function $u_\gamma(\cdot)$ is Lipschitz continuous on $U$.
It is not easy to calculate the Bouligand subdifferential
of $u_\gamma(\cdot)$ although locally it is almost everywhere differentiable.
Instead, it is more reasonable to consider using the composite mapping $\cT_{u_{\gamma}}:\Re^n\rightrightarrows \Re^{n\times n}$ defined by
\[
\label{surjacobian}
\begin{array}{ll}
\cT_{u_{\gamma}}(x)=\partial_B \cP_{{\gamma}^{-1}q} (\cdot)\mid_{x-{\gamma}^{-1}F(x)}\cdot \left(I-{\gamma}^{-1} F'(x)\right)-I.
\end{array}
\]
Note that for any $\gamma>0$ the mapping $\cT_{u_{\gamma}}(\cdot)$ defined in \eqref{surjacobian} is outer semicontinuous at $\bar x$.

Next, we show that $u_\gamma$ ($\gamma>0$)
defined in \eqref{proximal} is G-semismooth at a solution $\bar x$ to the GE \eqref{gecp} with respect to $\cT_{u_{\gamma}}$ defined by \eqref{surjacobian}. For convenience of comparison, we take identical values for all parameters here to those used in the conditions in \cite{Gfrerer222}.
\begin{proposition}
\label{ssbound1}
Let $\bar x$ be a solution to the GE \eqref{gecp}, and
$\mathbb B_{r}(\bar x)$ be the ball such that $F$ is Lipschitz continuous on it with modulus $\ell\ge 0$.
For any $\epsilon>0$, let $\delta$ and $\delta_q$ be two positive constants (depending on $\epsilon$) such that
$\delta\le \min\{\frac{\delta_q}{1+\ell}, r\}$,
\[
\label{scccite}
\begin{array}{rr}
|\langle e^*, d-\bar x\rangle
-\langle e, d^*+F(\bar x)\rangle|
\le
\frac{\epsilon}{2\sqrt{2}(\ell+1)}\|(e;e^*)\|\|(d-\bar x; d^*+F(\bar x))\|
~~~~
\\[2mm]
\forall (d;d^*)\in\gph\partial q\cap
\mathbb B_{\delta_q}(\bar x; -F(\bar x)),\quad \forall (e;e^*)\in L\in \cS^{*}_{\partial q}(d,d^*),
\end{array}
\]
and
\[
\label{ftylor}
\begin{array}{ll}
\|
F(x)-F(\bar x)- F'(x)(x-\bar x)
\|
\le \frac{\epsilon}{2\sqrt{2}(\ell+1)}\|x-\bar x\|\quad\forall x\in \mathbb B_{\delta}(\bar x).
\end{array}
\]
Then, for any $\hat x\in \Re^n$ satisfying
$\|\hat x-\bar x\|\le
\min\big\{\frac{\min\{\delta_q,r\}}{(1+\frac{\ell}{\gamma})\max\{1,\gamma\}},\delta\big\}$ the following results hold:
\begin{description}
\item[(a)]
For any $v\in\Re^n$ and any $B \in \partial_B \cP_{{\gamma}^{-1}q} (\hat x-{\gamma}^{-1}F(\hat x))$, one has
\[
\label{scd_needed}
\begin{array}{l}
\left| \langle \gamma v,u_\gamma(\hat x)\rangle -
\left(\langle \gamma B v, (I-{\gamma}^{-1} F'(\hat x ))(\hat x- \bar{x})\rangle-\langle \gamma v,\hat x-{\bar x}\rangle\right)
\right|
\\[1mm]
\leq \frac{\epsilon}{2\sqrt{2}(l+1)}
\big(\|( B v; \gamma( I-B)v)\|
\max\{1,\gamma\}(1+\frac{\ell}{\gamma})
+\|Bv\|\big)\|\hat x-\bar x\|.
\end{array}
\]
\item[(b)]
For any $B\in \partial_B \cP_{{\gamma}^{-1}q} (\hat x-{\gamma}^{-1}F(\hat x))$ such that $C : =B(I-{\gamma}^{-1} F'(\hat x))-I\in \cT_{u_\gamma}(\hat x)$ is nonsingular, by taking
$ M := \gamma C^\top $ one has
$$
\begin{array}{ll}
\vert| C^{-1}u_\gamma(\hat x) - (\hat x-\bar{x})|\vert
\le \frac{\epsilon}{2\sqrt{2}(\ell+1)}
\big(
\max\{1,\gamma\}(1+\frac{\ell}{\gamma})\|( B {M}^{-1}; \gamma( I-B){M}^{-1})\|_F
+\|B{M}^{-1}\|_F
\big)
\|\hat x-\bar x\|.
\end{array}
$$

\end{description}
\end{proposition}
\begin{proof}{Proof}
{\bf (a)}
Note that whenever $\hat x$ satisfies 
$\|\hat x-\bar x\|\le \frac{\min\{\delta_q,r\}}{(1+\frac{l}{\gamma})\max\{1,\gamma\}}$, 
one can get from \Cref{prop:uineq} that
\begin{equation*}
\begin{array}{l}
\|
(\hat y-\bar x;\gamma(\hat z-\hat y)+F(\bar x))
\|
=
\|(\hat u+\hat x-\bar x; F(\bar x) -F(\hat x)-\gamma \hat u)\|
\\[1mm]
\le \max\{1,\gamma\} \|\hat z- \bar z\|
\le \max\{1,\gamma\}(1+\frac{\ell}{\gamma})\|\hat x-\bar x\|
\le \min\{\delta_q,r\},
\end{array}
\end{equation*}
where $\hat u:=u_\gamma(\hat x)$, $\hat z:=\hat x-{\gamma}^{-1}F(\hat x)$,
$\hat y:=\hat u+\hat x=\cP_{{\gamma}^{-1}q} (\hat z)$,
and
$\bar z:=\bar x-{\gamma}^{-1}F(\bar x)$.
Therefore, whenever $\hat x$ satisfies $\|\hat x-\bar x\|\le \frac{\min\{\delta_q,r\}}{(1+\frac{l}{\gamma})\max\{1,\gamma\}}$, one can take
$(d;d^*)=(\hat y;\gamma(\hat z-\hat y))
\in\gph\partial q\cap \mathbb B_{\delta_q}(\bar x, -F(\bar x))$ in \eqref{scccite}
such that for all $(e;e^*)\in L\in \cS^{*}_{\partial q}(d,d^*)$,
\[
\label{scccite2}
\begin{array}{ll}
|\langle \gamma e, \hat z-\bar z\rangle
-\langle e^*+\gamma e, \hat y-\bar x \rangle |
\\[1mm]
=
|\langle e^*+\gamma e, \hat y-\bar x\rangle
-\langle \gamma e, \hat z-\bar x + \frac{1}{\gamma } F(\bar x)\rangle|
=|\langle e^*, \hat y-\bar x\rangle
-\langle e, \gamma(\hat z-\hat y)+F(\bar x)\rangle|
\\[1mm]
\le
\frac{\epsilon}{2\sqrt{2}(\ell+1)}\|(e;e^*)\|\|(\hat y-\bar x; \gamma(\hat z-\hat y)+F(\bar x))\|
\le
\frac{\epsilon}{2\sqrt{2}(\ell+1)}\|(e;e^*)\| \max\{1,\gamma\}(1+\frac{\ell}{\gamma})\|\hat x-\bar x\|.
\end{array}
\]
Also, by using \eqref{spq} of \Cref{lem:spq} one can get
 
$$
\cS^{*}_{\partial q}(\hat y,\gamma(\hat z-\hat y))=\{\range (-{\gamma}^{-1}B;-(I-B))\mid B\in\partial_B\cP_{{\gamma}^{-1}q} (\hat z) \}
=\{ \range(B;\gamma(I-B))\mid B\in\partial_B\cP_{{\gamma}^{-1}q} (\hat z) \}.
$$
 
Thus, $(e ;e^*)\in L$ if and only if $(e ;e^*)=(B v;\gamma( I-B)v)$ for some $v\in\Re^n$, where $B\in \partial_B \cP_{{\gamma}^{-1}q} (\hat z)$ is the symmetric positive semidefinite $n\times n$ matrix such that 
$L=\range(B;\gamma (I-B))$. 
 Moreover, in this case, one has from \eqref{spq} that
$
(e^*+\gamma e;
\gamma e)=(\gamma( I-B)v+\gamma B v; \gamma B v )
=
\gamma(v; B v )\in \cS^{*}_{\cP_{{\gamma}^{-1}q} }(z,y)$.

Therefore, if $
\hat x$ satisfies $\|\hat x-\bar x\|\le \frac{\min\{\delta_q,r\}}{(1+\frac{l}{\gamma})\max\{1,\gamma\}}$, one can obtain from \eqref{scccite2} that for any $v\in\Re^n$ and any $B\in \partial_B \cP_{{\gamma}^{-1}q} (\hat x-{\gamma}^{-1}F(\hat x))$ it holds that
\[
\label{Gss_1}
\begin{array}{ll}
|\langle \gamma Bv, \hat z-\bar z\rangle
-\langle \gamma v, \hat y-\bar x \rangle |
\le
\frac{\epsilon}{2\sqrt{2}(\ell+1)}\|(B v;\gamma( I-B)v)\|\max\{1,\gamma\}(1+\frac{\ell}{\gamma})\|\hat x-\bar x\|.
\end{array}
\]
Note that both \eqref{ftylor} and \eqref{Gss_1} hold if $\hat x$ satisfies
$\|\hat x-\bar x\|\le
\min\big\{\frac{\min\{\delta_q,r\}}{(1+\frac{\ell}{\gamma})\max\{1,\gamma\}},\delta\big\}
\le
\min\big\{\frac{\delta_q}{1+\ell}, r\big\}$.
In this case, one has
$$
\begin{array}{ll}
\left| \langle \gamma v,u_\gamma(\hat x)\rangle -
\left(\langle \gamma B v,(I-{\gamma}^{-1} F'(\hat x ))(\hat x- \bar{x})\rangle-\langle \gamma v,\hat x-{\bar x}\rangle\right)
\right|
\\[1mm]
=\vert \langle \gamma v,\hat u\rangle
+ \langle \gamma v,\hat x-{\bar x}\rangle
- \langle \gamma Bv,\hat x- \bar{x}
-{\gamma}^{-1}(F(\hat x)- F(\bar{x}))
-{\gamma}^{-1}F'(\hat x)(\hat x-\bar x)
+{\gamma}^{-1}(F(\hat x)- F(\bar{x}))\rangle|
\\[1mm]
\leq\vert \langle \gamma v,\hat y-{\bar x}\rangle-
\langle \gamma Bv, \hat z- \bar{z}\rangle\vert
+ |\langle B v, F(\hat x)-F(\bar{x})- F'(\hat x)(\hat x-\bar{x}) \rangle|
\\[1mm]
\leq \frac{\epsilon}{2\sqrt{2}(\ell+1)}
\|( B v;\gamma( I-B)v)\|
\max\{1,\gamma\}(1+\frac{\ell}{\gamma})\|\hat x-\bar x\|+ \frac{\epsilon}{2\sqrt{2}(\ell+1)}\|Bv\|\|\hat x-\bar x\|,
\end{array}
$$
\normalsize
which completes the proof of (a).

{\bf(b)}
When $C$ is not singular, we can take $v_i$ in \eqref{scd_needed} as the $i$-th column of $M^{-1}=(\gamma C^\top )^{-1}$ for $i=1,...,n$, that is,
$\gamma v_i^\top $ is the $i$-th row of $C^{-1}$. Consequently,
\begin{equation*}
\begin{array}{l}
\|C^{-1} u_\gamma(\hat x) - (\hat x-\bar{x})\|
\\ 
=
\| C^{-1}u_\gamma(\hat x)-C^{-1}B(I- {\gamma}^{-1} F'(\hat x))(\hat x- \bar{x})+C^{-1}(\hat x-{\bar x})\|
\\ 
=\big(
\sum_{i=1}^n\left| \langle \gamma v_i,u_\gamma(\hat x)\rangle -
\left(\langle \gamma B v_i, (I-{\gamma}^{-1} F'(\hat x ))(\hat x- \bar{x})\rangle-\langle \gamma v_i,\hat x-{\bar x}\rangle\right)
\right|^2\big)^{\frac{1}{2}}
\\ 
\le
\frac{\epsilon}{2\sqrt{2}(\ell+1)}
\Big(\sum_{i=1}^n
\left(
\max\{1,\gamma\}(1+\frac{\ell}{\gamma})
\|( B v_i; \gamma( I-B)v_i)\|
+\|Bv_i\|\right)^2
\Big)^{\frac{1}{2}}
\|\hat x-\bar x\|
\\ 
\le
\frac{\epsilon}{2\sqrt{2}(\ell+1)}
\big(
\max\{1,\gamma\}(1+\frac{\ell}{\gamma})\big(\sum_{i=1}^n
\|( B v_i; \gamma( I-B)v_i)\|^2
\big)^{\frac{1}{2}}
+\big(\sum_{i=1}^n\|Bv_i\|^2 \big)^{\frac{1}{2}}
\big)
\|\hat x-\bar x\|
\\[1mm]
=
\frac{\epsilon}{2\sqrt{2}(\ell+1)}
\big(
\max\{1,\gamma\}(1+\frac{\ell}{\gamma})\|( B {M}^{-1}; \gamma( I-B){M}^{-1})\|_F
+\|B{M}^{-1}\|_F
\big)
\|\hat x-\bar x\|,
\end{array}
\end{equation*}
where the last inequality comes from the triangle inequality and the final equality comes from the definition of the Frobenius norm. This completes the proof of (b).
 \end{proof}

The G-semismoothness of $u_\gamma$ with respect to $\cT_{u_\gamma}$ is given as follows.

\begin{corollary}
Let $\bar x$ be a solution to the GE \eqref{gecp}, and
$\mathbb B_{r}(\bar x)$ be the ball such that $F$ is Lipschitz continuous on it with modulus $\ell\ge 0$.
Assume that $\partial q$ is SCD semismooth$^*$ at $(\bar{x},-F(\bar{x}))$.
For any $\gamma>0$ the mapping $u_\gamma$ defined in \eqref{proximal} is G-semismooth with respect to $\cT_{u_\gamma }$ given in \eqref{surjacobian} at $\bar x$.
\end{corollary}

\begin{proof}{Proof}
Let $\epsilon>0$ be arbitrarily given.
Note that one can find two positive constants
$\delta$ and $\delta_q$ with
$\delta\le \min\{\frac{\delta_q}{1+\ell}, r\}$
such that \eqref{scccite} and \eqref{ftylor} hold.
Then by \Cref{ssbound1}(a) we know that
for any $\hat x\in \Re^n$ satisfying
$\|\hat x-\bar x\|\le \min\big\{\frac{\min\{\delta_q,r\}}{(1+\frac{\ell}{\gamma})\max\{1,\gamma\}},\delta\big\}$,
it holds for any $v\in\Re^n$ and any $B\in \partial_B \cP_{{\gamma}^{-1}q} (\hat x-{\gamma}^{-1}F(\hat x))$ that
\begin{equation*}
\begin{array}{l}
\left| \langle \gamma v,u_\gamma(\hat x)\rangle -
\left(\langle \gamma B v, (I-{\gamma}^{-1} F'(\hat x ))(\hat x- \bar{x})\rangle-\langle \gamma v,\hat x-{\bar x}\rangle\right)
\right|
\\[1mm]
\leq \frac{\epsilon}{2\sqrt{2}(\ell+1)}
\big(\|( B v; \gamma( I-B)v)\|
\max\{1,\gamma\}(1+\frac{\ell}{\gamma})
+\|Bv\|\big)\|\hat x-\bar x\|.
\end{array}
\end{equation*}
Since $B$ is a firmly nonexpansive mapping by \cite[Proposition 3.22]{gfrerer2022local}, one can get that
$\|( B v; \gamma( I-B)v)\| \leq \max\{1,\gamma\}\| v\|$.
Therefore, by taking $v$ as the vector such that $\|\gamma v\|=1$ and
\begin{equation*}
\begin{array}{l}
\langle \gamma v, u_\gamma(\hat x)-(B(I-{\gamma}^{-1} F'(\hat x ))-I)(\hat x- \bar{x}) \rangle = \|u_\gamma(\hat x)-(B(I-{\gamma}^{-1} F'(\hat x ))-I)(\hat x- \bar{x})\|,
\end{array}
\end{equation*}
one can get $\|Bv\|\le \max\{1,\frac{1}{\gamma}\}$ and 
\begin{equation*}
\begin{array}{ll}
\|u_\gamma(\hat x)-(B(I-{\gamma}^{-1} F'(\hat x ))-I)(\hat x- \bar{x})\|
&\leq
\frac{\epsilon}{2\sqrt{2}(\ell+1)}
\big(\max\{1,\gamma^2\}\| v\|(1+\frac{\ell}{\gamma})
+\|Bv\|\big)\|\hat x-\bar x\|.
\end{array}
\end{equation*}
Consequently, $u_\gamma$ is G-semismooth with respect to
$\cT_{u_\gamma}$, and this completes the proof.
 \end{proof}

\subsection{Regularity conditions}
When using $\cT_{u_\gamma}$ in \eqref{surjacobian} as a generalized Jacobian in a G-semismooth Newton method for solving the nonsmooth equation $u_\gamma(x)=0$,
conditions for ensuring $\cT_{u_\gamma}$ being nonsingular around a reference solution should be verified.
On the other hand, Algorithm \ref{algo:ssnm_GE} is well-defined only if the coefficient matrix of the linear equation \eqref{iteration} is nonsingular.
In fact, we have the following results on the equivalence between the two regularity conditions mentioned above.
\begin{lemma}
\label{nonsingularequiv}
Let $\gamma>0$ and $u_{\gamma}$ be the function defined in \eqref{proximal}.
For any $ x\in\Re^n$,
every element of $\cT_{u_{\gamma}}( x)$ is nonsingular if and only if
${(Y^{*}}^\top F'( x)+{X^{*}}^\top )$ is nonsingular for all $X^*\in\Re^{n\times n}$ and $Y^*\in\Re^{n\times n}$ such that
$\range (Y^{*};X^{*})\in\cS^{*}_{\partial q}\big( x+u_{\gamma}( x),-\gamma u_\gamma( x)-F( x) \big)$.
\end{lemma}
\begin{proof}{Proof}
For convenience, denote
$ u:=u_\gamma( x)$, $ z:= x-{\gamma}^{-1}F(x)$,
and $ y:= u+ x=\cP_{{\gamma}^{-1}q} ( z)$.
It is easy to see from \eqref{proxdef} that
$\gamma( z- y)\in\partial q( y)$,
so that
$\big( x+u_{\gamma}( x);-\gamma u_\gamma( x)-F( x) \big)=( y; \gamma( z- y))\in\gph \partial q$.
 By using \eqref{spq} of \Cref{lem:spq} we have 
\begin{equation}
\label{partial_q_der}
\cS^*_{\partial q}\big( x+u_{\gamma}( x),-\gamma u_\gamma( x)-F( x) \big)
=\big\{
\range (B;\gamma(I-B)) \mid
B\in\partial_B\cP_{{\gamma}^{-1}q} ( z)
\big\}.
\end{equation}
By \eqref{surjacobian} one has
$
\cT_{u_{\gamma}}( x)=\{B(I-{\gamma}^{-1} F'( x))-I
\mid
B\in \partial_B \cP_{{\gamma}^{-1}q} ( z)\}
$.
Thus, it is sufficient to prove that for every $B\in \partial_B \cP_{{\gamma}^{-1}q} ( z)$, the matrix
$B(I-{\gamma}^{-1} F'( x))-I$ is nonsingular if and only if
${ Y^{*}}^\top F'( x)+{X^{*}}^\top $ is nonsingular
for all $X^*\in\Re^{n\times n}$ and $Y^*\in\Re^{n\times n}$ such that
$\range(Y^{*};X^{*})=\range(B;\gamma(I-B))$.

Fix $B\in \partial_B \cP_{{\gamma}^{-1}q} ( z)$.
On the one hand, suppose that
$B(I-{\gamma}^{-1} F'( x))-I\in \cT_{u_{\gamma}}( x)$ is singular.
By taking
$Y^{*}=B$ and $X^{*}=\gamma(I-B)$
one has
$
{Y^{*}}^\top F'( x)+{X^{*}}^\top =B F'( x)+\gamma(I-B)
=-\gamma(B(I-{\gamma}^{-1} F'( x))-I)$, 
which is also singular.
On the other hand, if ${Y^{*}}^\top F'( x)+{X^{*}}^\top $ is singular with $\range(Y^{*};X^{*})=\range(B;\gamma(I-B))$, i.e., there exists a nonzero vector $v\in\Re^n$ such that
${Y^{*}}^\top F'( x)v+{X^{*}}^\top v=0$, one has for every $w\in\Re^n$, $$
w^\top {Y^{*}}^\top F'( x)v+w^\top {X^{*}}^\top v=
(Y^{*}w)^\top F'( x)v+({X^{*}w})^\top v=
\left((Y^{*}w)^\top F'( x)+({X^{*}w})^\top \right)v=0.
$$
Since $\range(Y^{*};X^{*})=\range(B;\gamma(I-B))$, one has
\begin{equation*}
w^\top \left( B F'( x)+ {\gamma(I-B)} \right)v
=
\big((Bw)^\top F'( x)+({\gamma(I-B)w})^\top \big)v=0\quad \forall w\in\Re^n.
\end{equation*}
Therefore, we have
$\left( B F'( x)+ {\gamma(I-B)} \right)v=0$, which implies
that the matrix $-\gamma(B(I-{\gamma}^{-1} F'( x))-I)$ is nonsingular. This completes the proof.
 \end{proof}

\begin{corollary}
Let $\bar x$ be a solution to the GE \eqref{gecp}.
Then, $F+\partial q$ is SCD regular around $(\bar{x},0)$ if and only if every element of $\cT_{u_\gamma}(\bar x)$ defined in \eqref{surjacobian} is nonsingular for any $\gamma>0$.
\end{corollary}
\begin{proof}{Proof}
According to \cite[Proposition 5.1(ii)]{Gfrerer222} we know that
$F+\partial q$ is SCD regular around $(\bar{x},0)$ if and only if
 ${Y^{*}}^\top F'(\bar x)+{X^{*}}^\top $ is nonsingular for all $X^*\in\Re^{n\times n}$ and $Y^*\in\Re^{n\times n}$ such that
$\range(Y^{*};X^{*})\in\cS^{*}_{\partial q}\big( \bar x,-F(\bar x) \big)$. Thus, the conclusion follows from \Cref{nonsingularequiv}. This completes the proof.
 \end{proof}

The following result is also a consequence of \Cref{nonsingularequiv}.
\begin{lemma}
\label{C_B}
Let $u_{\gamma}$ $(\gamma>0)$ be the function defined in \eqref{proximal}, and $\cJ$ be the mapping defined in \eqref{GExd}.
For any $ x\in \Re^n$ such that every element of $\cT_{u_{\gamma}}( x)$ is nonsingular, one has
\[
\label{sfcb}
\begin{array}{ll}
\cS^*_{\cJ}
\big(
( x, x+u_{\gamma}( x)),
(F( x)+d^*,-u_\gamma( x))
\big)
=\big\{\range\left(C_B; I\right)
\mid 
\begin{array}{ll}
B\in \partial_B \cP_{{\gamma}^{-1}q} ( x-{\gamma}^{-1}F( x))
\end{array}
\big\}
\end{array},
\]
where
$
C_{B}:= \begin{pmatrix}
B\widetilde M^{-1} & B \widetilde M^{-1}
\\
\gamma(I-B)\widetilde M^{-1} & \gamma(I-B)\widetilde M^{-1}-I
\end{pmatrix}$
with
$\widetilde M:=F'( x)^\top B+\gamma(I-B)$.
\end{lemma}
\begin{proof}{Proof}
According to \cite[Proposition 5.1(1)]{Gfrerer222} we know that, for any $
(x;d)\in\Re^{2n}$ and $d^*\in\partial q(d)$,
$$
\begin{array}{ll}
\cS^*_\cJ((x,d),(F(x)+d^*,x-d))
=\left\{\range\left(\begin{pmatrix}
Y^* & 0
\\
0 & -I
\end{pmatrix}
\ ; \ \begin{pmatrix}
F'(x)^\top Y^* & -I
\\
X^* &\ I
\end{pmatrix}\right)
\Big |\,
\range(Y^{*};X^{*})\in\cS^{*}_{\partial q}(d,d^{*})
\right\}.
\end{array}
$$
From \Cref{nonsingularequiv} we know that
every element of $\cT_{u_{\gamma}}( x)$ is nonsingular if and only if
${(Y^{*}}^\top F'( x)+{X^{*}}^\top )$ is nonsingular for all $X^*$ and $Y^*$ such that
$\range(Y^{*};X^{*})\subseteq
\cS^{*}_{\partial q}\big( x+u_{\gamma}( x),-\gamma u_\gamma( x)-F( x) \big)$.
Moreover, \eqref{partial_q_der} holds in this case, so that with $ z:= x-{\gamma}^{-1}F( x)$ one has
\[
\label{cftemp}
\begin{array}{ll}
\cS^*_{\cJ}\big(( x, x+u_{\gamma}( x)),
(F( x)+d^*,-u_\gamma( x))\big)
\\[2mm]
=\left\{
\range\left(\begin{pmatrix}
Y^* & 0
\\
0 & -I
\end{pmatrix}
\begin{pmatrix}
F'( x)^\top Y^* & -I
\\
X^* &\ I
\end{pmatrix}^{-1}; I\right)
\Big |
\begin{array}{ll}
\range(Y^{*};X^{*})= \range(B;\gamma (I-B)),
\\[1mm]
B\in \partial_B \cP_{{\gamma}^{-1}q} ( z)
\end{array}
\right\}.
\end{array}
\]
Fix $B\in \partial_B \cP_{{\gamma}^{-1}q} ( z)$ and choose
$(Y^{*};X^{*})$ such that $\range(Y^{*};X^{*})= \range(B;\gamma (I-B))$.
Since that $(B;\gamma (I-B))$ has full column rank,
there exists a nonsingular matrix $P\in \Re^{n\times n}$ such that $Y^{*} = BP$ and $X^{*} = \gamma (I-B)P$. Moreover, since $M:=F'( x)^\top Y^*+X^*=(F'( x)^\top B+\gamma(I-B)) P$ is nonsingular, one has $M^{-1}=P^{-1}(F'( x)^\top B+\gamma(I-B))^{-1}$ and
$$
\begin{array}{l}
\begin{pmatrix}
Y^* & 0
\\
0 & -I
\end{pmatrix}
\begin{pmatrix}
F'( x)^\top Y^* & -I
\\
X^* &\ I
\end{pmatrix}^{-1}
=
\begin{pmatrix}
B P & 0
\\
0 & -I
\end{pmatrix}
\begin{pmatrix}
F'( x)^\top B P & -I
\\
\gamma (I-B) P &\ I
\end{pmatrix}^{-1}
\\[4mm]
=\begin{pmatrix}
B P & 0
\\
0 & -I
\end{pmatrix}
\begin{pmatrix}
M^{-1} & M^{-1}
\\
-\gamma (I-B) P M^{-1} &\ I- \gamma (I-B) P M^{-1}
\end{pmatrix}
=
\begin{pmatrix}
B P M^{-1} & B P M^{-1}
\\
\gamma(I-B) P M^{-1} &\ \gamma(I-B) P M^{-1}-I
\end{pmatrix}
\\[5mm]
=
\begin{pmatrix}
B (F'( x)^\top B+\gamma(I-B))^{-1} & B (F'( x)^\top B+\gamma(I-B))^{-1}
\\
\gamma(I-B) (F'( x)^\top B+\gamma(I-B))^{-1} &\ \gamma(I-B) (F'( x)^\top B+\gamma(I-B))^{-1} -I
\end{pmatrix}.
\end{array}
$$
Therefore, one can remove the dependence on the precise choice of $(Y^*; X^*)$ and use only the information of $B$ in \eqref{cftemp}. So we obtain \eqref{sfcb}. This completes the proof.
 \end{proof}

\subsection{Equivalence to a G-semismooth Newton method}
Based on the discussions in the previous subsections, a G-semismooth Newton method (Algorithm \ref{alg:nm})
for solving \eqref{gecp} via its equivalent form \eqref{proximal} can be given as follows.

\begin{algorithm}[H]
\caption{A G-semismooth Newton method for solving \eqref{gecp} \label{alg:gssnge}}

\Input {$x^{(0)}\in\Re^n$, $F:\Re^n\to\Re^n$, and $q:\Re^n\to\bar{\Re}$.} 
\Output{$\{x^{(k)}\}$.}
\For{$k=0,1, \ldots$,}{
1. if $0\in F(x^{(k)})+\partial q(x^{(k)})$, stop the algorithm;
\\[1mmm]
2. select $\gamma^{(k)}>0$, and compute $\text{ }u^{(k)}:=u_{\gamma^{(k)}}(x^{(k)})$;
\\[1mmm]
3. select
$V^{(k)}\in
\cT_{u_{\gamma^{(k)}}}(x^{(k)})$ via \eqref{surjacobian}, then compute the Newton direction $\Delta x^{(k)}$ from
$V^{(k)} \Delta x^{(k)}=-u^{(k)}$,
and obtain the new iterate via $x^{(k+1)}:= x^{(k)}+\Delta x^{(k)}$.
}
\end{algorithm}

Note that Algorithm \ref{alg:gssnge} is essentially a ``uniform'' version of Algorithm \ref{alg:nm} for solving a family of problems, i.e., $\{u_\gamma (x)=0,\gamma>0\}$, sharing the common solutions.
In each iteration, one selects one instance of these problems and performs the G-semismooth Newton step via $\cT_{u_\gamma}$. 
 In the following, we show that Algorithm \ref{alg:gssnge} is well-defined if 
Algorithm \ref{algo:ssnm_GE} is, and vice versa.
Moreover, a sequence generated by one of them can be treated as the one generated by the other.

\begin{lemma}
\label{prop:equiv1}
Given $x^{(0)}\in \Re^n$ and $\{\gamma^{(k)}\}$.
 Suppose that both Algorithm \ref{algo:ssnm_GE} and Algorithm \ref{alg:gssnge} generate the same point $x^{(\bar k)}$ after the iteration indexed by $(\bar k-1)\ge 0$, and $B^{(\bar k)}\in \partial_B \cP_{{\gamma^{(\bar k)}}^{-1} q} (\cdot)\mid_{x^{(\bar k)}-{\gamma^{(\bar k)}}^{-1}F(x^{(\bar k)})}$ is chosen such that
\begin{equation}
\label{fixmatrix}
\range(Y^{*(\bar k)};X^{*(\bar k)}):=\range(B^{(\bar k)};\gamma^{(\bar k)}(I-B^{(\bar k)}))
\quad\mbox{and}\quad 
V^{(\bar k)}:=B^{(\bar k)}(I- {\gamma^{(\bar k)}}^{-1} F'( x^{(\bar k)}))-I.
\end{equation}
Then, one has $\range(Y^{*(\bar k)};X^{*(\bar k)})\in \cS^{*}_{\partial q}\big(\hat{d}^{(\bar k)},\hat{d}^{*(\bar k)}\big) $ 
and 
$V^{(\bar k)}\in \cT_{u_{\gamma^{(\bar k)}}}(x^{(\bar k)})$.
Moreover, both the two algorithms generate the same $x^{(\bar k+1)}$ if \eqref{fixmatrix} is used in them.
\end{lemma}
 
\begin{proof}{Proof}
Recall that, in iteration indexed by $\bar k$ of Algorithm \ref{algo:ssnm_GE}, the Newton direction $\Delta x^{(\bar k)}$ generated by \eqref{iteration}, i.e. the linear system
\begin{equation}
\label{iteration_SCD}
 {(Y^{*(\bar k)}}^\top F'(x^{(\bar k)})+{X^{*(\bar k)}}^\top )\Delta x^{(\bar k)}={( \gamma^{(\bar k)}Y^{*(\bar k)}}^\top +{X^{*(\bar k)}}^\top )u^{(\bar k)},
\end{equation}
where $u^{(\bar k)}:=u_{\gamma^{(\bar k)}}(x^{(\bar k)})$, $\hat{d}^{(\bar k)}:=x^{(\bar k)}+u^{(\bar k)}$ and $\hat{d}^{*(\bar k)}:=-\gamma^{(\bar k)}u^{(\bar k)}-F(x^{(\bar k)})$.
Thus, by denoting $W$ as the nonsingular transition matrix such that $Y^{*(\bar k)}=B^{(\bar k)}W$ and $X^{*(\bar k)}=\gamma^{(\bar k)}(I-B^{(\bar k)})W$, \eqref{iteration_SCD} is equivalent to
\begin{equation*}
 {(W^\top {B^{(\bar k)}}} F'(x^{(\bar k)})+{\gamma^{(\bar k)} W^\top (I-B^{(\bar k)}))\Delta x^{(\bar k)}=( \gamma^{(\bar k)} W^\top B^{(\bar k)}}+{\gamma^{(\bar k)} W^\top (I-B^{(\bar k)})})u^{(\bar k)},
\end{equation*}
or equivalently,
$\big( B^{(\bar k)} \big(I- \frac{1}{\gamma^{(\bar k)} }F'(x^{(\bar k)})\big)-I\big)\Delta x^{(\bar k)}= - u^{(\bar k)}$,
which is exactly the Newton system of the G-semismooth Newton method using $\cT_{u_{\gamma^{(\bar k)}}}$ (Algorithm \ref{alg:gssnge}). This completes the proof.
 \end{proof}

Next, we show that the local superlinear convergence of Algorithm \ref{alg:gssnge}
can be obtained under the assumptions made in \Cref{ThConvSSNewtonVI} (i.e., \cite[Theorem 5.2]{Gfrerer222}).
\begin{proposition}
\label{u_gamma_Gss}
Let $\bar{x}$ be a solution of \eqref{gecp}. Assume that $\partial q$ is SCD semismooth$^*$ at $(\bar{x},-F(\bar{x}))$
and $\partial q+F$ is SCD regular around $(\bar{x},0)$.
Then for every pair $\underline{\gamma},\bar\gamma$ with $0<\underline{\gamma}\leq \bar\gamma$ there exists a neighborhood $\cU$ of $\bar{x}$ such that for every starting point $x^{(0)}\in \cU$ Algorithm \ref{alg:gssnge} produces a sequence $\{x^{(k)}\}$ converging superlinearly to $\bar{x}$, provided we choose in every iteration step $\gamma^{(k)}\in [\underline{\gamma},\bar\gamma]$.
\end{proposition}

 \begin{proof}{Proof}
From \cite[Proposition 5.1]{Gfrerer222} we know $\partial q+F$ is SCD regular around $(\bar{x},0)$ if and only if the mapping $\cJ$ defined in \eqref{GExd} is SCD regular around $\left((\bar{x},\bar{x}),(0,0)\right)$. Moreover, according to \cite[Proposition 4.8]{gfrerer2022local},
for every $\kappa> {\rm scd\, reg}\cJ((\bar{x},\bar{x}),(0,0))$, one can find a positive radius $\rho>0$ such that
for every $\gamma \in [\underline{\gamma},\overline{\gamma}]$ and every $x\in \Re^n$ such that $(x,x + u_{\gamma}(x)) \in \mathbb B_\rho(\bar x,\bar x)$, the mapping $\cJ$ is also SCD regular
around $((x,x + u_{\gamma}(x)),(-\gamma u_{\gamma}(x),-u_{\gamma}(x)))$
and
$\kappa>{\rm scd\, reg}\, \cJ((x,x + u_{\gamma}(x)),(-\gamma u_{\gamma}(x),-u_{\gamma}(x)))$.
Thus, by \Cref{nonsingularequiv} and \cite[Proposition 5.1(1)]{Gfrerer222}, each element of $\cT_{u_\gamma}(x)$ is nonsingular. Moreover, by combining \cite[eq. (34)]{gfrerer2022local} and \Cref{C_B}, we obtain for any $B \in \partial_B \cP_{{\gamma}^{-1}q} (\cdot)\mid_{x-{\gamma}^{-1}F(x)}$,
\begin{equation} \label{bound_kappa}
\left\|
\begin{pmatrix}
B[F'(x)^\top B+ {\gamma}(I-B)]^{-1} & B [F'(x)^\top B+ {\gamma}(I-B)]^{-1}
\\[1mm]
\gamma(I-B)[F'(x)^\top B+ {\gamma}(I-B)]^{-1} & \gamma(I-B)[F'(x)^\top B+ {\gamma}(I-B)]^{-1}-I
\end{pmatrix}
\right\|
\leq \kappa.
\end{equation}
Let $\mathbb B_{r}(\bar x)$ be the ball such that $F$ is Lipschitz continuous on it with modulus $\ell\ge 0$.
Take $0<\xi<1$ and
$$
\begin{array}{ll}
\epsilon := \frac{2\sqrt{2}(\ell+1)\xi}{\sqrt{n}\kappa
\mathop{\max}
\limits_{\gamma \in[\underline{\gamma},\overline{\gamma}]}\left(\max \{1,\gamma\} (1+\frac{\ell}{\gamma})+1\right)}.
\end{array}
$$
Then, for any $\hat x$ satisfying
$\|\hat x-\bar x\|\le
\min\big\{\frac{\min\{\delta_q,r\}}{(1+\frac{\ell}{\gamma})\max\{1,\gamma\}},\delta\big\}$ and $(\hat x,\hat x + u_{\gamma}(\hat x)) \in \mathbb B_\rho(\bar x,\bar x)$, 
where $\delta_q$ and $\delta\le \min\{\frac{\delta_q}{1+\ell}, r\}$ are the positive constants (depending on $\epsilon$) such that \eqref{scccite} and \eqref{ftylor} holds, 
one can obtain from \Cref{ssbound1}(b) the nonsingularity of each element $\widehat C : = \widehat B(I-{\gamma}^{-1} F'(\hat x))-I\in \cT_{u_\gamma}(\hat x)$ with $\widehat B\in \partial_B \cP_{{\gamma}^{-1}q} (\hat x-{\gamma}^{-1}F(\hat x))$, and
\begin{equation*}
\begin{array}{ll}
\vert| \widehat C^{-1}u_\gamma(\hat x) - (\hat x-\bar{x})|\vert
\\[1mm]
\le \frac{\epsilon}{2\sqrt{2}(\ell+1)}
{\left(
\max\{1,\gamma\}(1+\frac{\ell}{\gamma})\|( \widehat B {M}^{-1}; \gamma( I-\widehat B){M}^{-1})\|_F
+\|\widehat B{M}^{-1}\|_F
\right)}
\|\hat x-\bar x\|
\\[1mm]
\le
\frac{\sqrt{n}\kappa\epsilon}{2\sqrt{2}(\ell+1)}\mathop{\max}
\limits_{\gamma \in[\underline{\gamma},\overline{\gamma}]}\left(\max \{1,\gamma\} (1+\frac{\ell}{\gamma})+1\right)
\|\hat x-\bar x\|=\xi \|\hat x-\bar x\|,
\end{array}
\end{equation*}
where $ M := \gamma \widehat C^\top \equiv -F'(\hat x)^\top \widehat B- {\gamma}(I-\widehat B)$, and the last inequality comes from \eqref{bound_kappa}.
Then by letting
$$
\begin{array}{ll}
\cU:=\Big\{\hat x \,\Big |\,
\|\hat x-\bar x\|\le
\min\big\{\frac{\min\{\delta_q,r\}}{(1+\frac{\ell}{\gamma})\max\{1,\gamma\}},\delta
\big\},
\ 
(\hat x,\hat x + u_{\gamma}(\hat x)) \in \mathbb{B}_\rho(\bar x,\bar x)\quad \forall \gamma \in[\underline{\gamma},\overline{\gamma}]
\Big\},
\end{array}
$$
one can get the convergence of Algorithm \ref{alg:gssnge} provided that $x^{(0)}\in \cU$.
The superlinear convergence rate comes from further shrinking the value of $\epsilon$, and this completes the proof.
 \end{proof}

Finally, we have the following result, showing that
Algorithm \ref{algo:ssnm_GE} is an
instance of Algorithm \ref{alg:gssnge}.

\begin{theorem}
Under the assumptions of
\Cref{ThConvSSNewtonVI},
Algorithm \ref{algo:ssnm_GE} is an instance of Algorithm \ref{alg:gssnge},
and the local superlinear convergence of Algorithm \ref{algo:ssnm_GE} (i.e., \Cref{ThConvSSNewtonVI}) can be guaranteed by \Cref{u_gamma_Gss}.
\end{theorem}
\begin{proof}{Proof} 
The conclusion of the theorem follows immediately from \Cref{prop:equiv1} and \Cref{u_gamma_Gss}.
 \end{proof}

\section{Conclusions}
\label{secconlu}
This paper showed that the two typical implementable semismooth* Newton methods are applications of G-semismooth Newton methods. 
This further enriches the comprehension of G-semismooth Newton methods and helps design practical Newton-type methods for GEs.
Accordingly, a natural question is whether an implementable semismooth* Newton method is achievable for solving a GE that cannot be reformulated to locally Lipschitz continuous equations.
Moreover, the relationship between the generic semismooth* Newton methods and G-semismooth Newton methods is still unclear, so another question is whether one can obtain generalizations of G-semismooth Newton methods, involving certain tractable ``approximation steps'', that can solve a broader class of problems. 
We leave these questions for future research.

\bibliographystyle{nonumber}

\begin{thebibliography}{3}
\providecommand{\natexlab}[1]{#1}
\providecommand{\url}[1]{\texttt{#1}}
\providecommand{\urlprefix}{URL }

\bibitem[{Arag\'on et~al.(2011)}]{Aragon15}
Arag\'on Artacho FJ, Dontchev AL, Gaydu M, Geoffroy MH, Veliov VM (2011)
Metric regularity of Newton’s iteration.
\emph{SIAM J. Control Optim.} 49(2):339-362.

\bibitem[{Arag\'on et~al.(2014)}]{aa2014}
Arag\'on Artacho FJ, Belyakov A, Dontchev AL, Lopez M (2014)
Local convergence of quasi-Newton methods under metric regularity. 
\emph{Comput. Optim. Appl.} 58(1):225--247.



\bibitem[{Arag\'on et~al.(2024)}]{Aragon24} 
Arag\'on Artacho FJ, Mordukhovich BS, P\'erez Aros P (2024) Coderivative-based semi-Newton method in nonsmooth difference programming. \emph{Math. Program.} 1--48.




\bibitem[{Az\'e and Chou(1995)}]{Aze95}
Az\'e D, Chou CC (1995) 
On a Newton type iterative method for solving inclusions. 
\emph{Math. Oper. Res.} 20(4):790--800.


\bibitem[{Bonnans(1994)}]{Bonans94}
Bonnans JF (1994) 
Local analysis of Newton-type methods for variational inequalities and nonlinear programming. 
\emph{Appl. Math. Optim.} 29:161--186.

\bibitem[{Bonnans and Shapiro(2000)}]{Perturbation}
Bonnans JF, Shapiro A (2000) 
\emph{Perturbation Analysis of Optimization Problems}
(Springer, New York).

\bibitem[{Cibulka et~al.(2011)}]{cibulka11}
Cibulka R, Dontchev AL, Geoffroy MH (2011) 
Inexact Newton methods and Dennis–Mor\'e theorems for nonsmooth generalized equations. 
\emph{SIAM J. Control Optim.} 49(2):339--362.



\bibitem[{Dias and Smirnov(2012)}]{Dias12}
Dias JM, Smirnov G (2012) 
On the Newton method for set-valued maps. 
\emph{Nonlinear Anal.} 75:1219--1230.

\bibitem[{Dontchev(1996{\natexlab{a}})}]{Dontchev961}
Dontchev AL (1996) 
Local convergence of the Newton method for generalized equations. 
\emph{Comptes Rendus de l'Acad\'emie des Sciences - Series I - Mathematics.}
322(4):327--331.


\bibitem[{Dontchev(1996{\natexlab{b}})}]{Dontchev96}
Dontchev AL (1996) 
Local analysis of a Newton-type method based on partial linearization. 
Renegar J, Shub M, and Smale S, eds. 
\emph{The mathematics of numerical analysis }
(AMS, Providence, RI) 295--306. 





\bibitem[{Dontchev and Rockafellar(2010)}]{Dontchev10}
Dontchev AL, Rockafellar RT (2010) 
Newton’s method for generalized equations: a sequential implicit function theorem. 
\emph{Math. Program.} 123(1):139--159.



\bibitem[{Ferreira(2015)}]{Ferreira15}
Ferreira OP (2015) 
A robust semi-local convergence analysis of Newton’s method for cone inclusion problems in Banach spaces under affine invariant majorant condition. 
\emph{J. Comput. Appl. Math.} 279:318--335.

\bibitem[{Ferreira and Silva(2017)}]{Ferreira17}
Ferreira OP, Silva GN (2017) Kantorovich’s theorem on Newton’s method for solving strongly regular generalized equation. 
\emph{SIAM J. Optim.} 27(2):910--926.

\bibitem[{Ferreira and Silva(2018)}]{Ferreira18}
Ferreira OP, Silva GN (2018) 
Local convergence analysis of Newton’s method for solving strongly regular generalized equations. 
\emph{J. Math. Anal. Appl.} 458(1):481--496.



\bibitem[{Fischer(1999)}]{Fischer99}
Fischer A (1999) 
Modified Wilson's method for nonlinear programs with nonunique multipliers. 
\emph{Math. Program.} 24:699--727.

\bibitem{GaoSun2009}
Gao Y, Sun DF (2009), 
Calibrating least squares semidefinite programming with equality
and inequality constraints.
\emph{SIAM J. Matrix Anal. Appl.} 31:1432--1457.


\bibitem[{Gaydu and Geoffroy(2013)}]{gaydu13}
Gaydu M, Geoffroy MH (2013) 
A Newton iteration for differentiable set-valued maps. 
\emph{J. Math. Anal. Appl.} 399(1):213--224.

\bibitem[{Geoffroy and Pi\'etrus(2004)}]{Geoffroy04}
Geoffroy MH, Pi\'etrus A (2004) Local convergence of some iterative methods for generalized equations. 
\emph{J. Math. Anal. Appl.} 290(2):497--505.

\bibitem[{Geoffroy and Pi\'etrus(2005)}]{Geoffroy05}
Geoffroy MH, Pi\'etrus A (2005) A general iterative procedure for solving nonsmooth generalized equations. 
\emph{Comput. Optim. Appl.} 31(1):57--67.



\bibitem{gfrerer2025}
Gfrerer H. (2025)
On a globally convergent semismooth* Newton method in nonsmooth nonconvex optimization.
\emph{Comput. Optim. Appl.} online, https://doi.org/10.1007/s10589-025-00658-z.


\bibitem[Gfrerer et al.(2023)]{gf23}
Gfrerer H, Mandlmayr M, Outrata JV, Valdman J (2023) 
On the SCD semismooth* Newton method for generalized equations with application to a class of static contact problems with Coulomb friction.
\emph{Comput. Optim. Appl.} 86:1159--1191.




\bibitem{calmref}
Gfrerer H., Mordukhovich BS (2019)
Second-order variational analysis of parametric constraint and variational systems. 
\emph{SIAM J. Optim.} 29(1):423--453.

\bibitem[{Gfrerer and Outrata(2021)}]{Gfrerer2021}
Gfrerer H, Outrata JV (2021) On a semismooth* Newton method for solving generalized equations. 
\emph{SIAM J. Optim.} 31(1):489--517.
 

\bibitem[{Gfrerer and Outrata(2022)}]{gfrerer2022local}
Gfrerer H, Outrata JV (2022) On (local) analysis of multifunctions via subspaces contained in graphs of generalized derivatives. 
\emph{J. Math. Anal. Appl.} 508(2):125895.


\bibitem[{Gfrerer et~al.(2022{\natexlab{a}})}]{Gfrerer222}
Gfrerer H, Outrata JV, Valdman J (2022) 
On the application of the SCD semismooth* Newton method to variational inequalities of the second kind. 
\emph{Set-Valued Var. Anal.} 30:1453--1484.






\bibitem[{Gowda(2004)}]{Gowda2004}
Gowda MS (2004) Inverse and implicit function theorems for H-differentiable and semismooth functions.
\emph{Optim. Methods Softw.} 19:443--461.



\bibitem[{Henrion(2001)}]{Henrion}
 Henrion R, Outrata JV (2001) A subdifferential condition for calmness of multifunctions. 
 \emph{J. Math. Anal. Appl.} 258:110--130.


\bibitem[{Hoheisel et~al.(2012{\natexlab{a}})}]{Hoheisel12}
Hoheisel T, Kanzow C, Mordukhovich BS, Phan H (2012) Generalized Newton’s method based on graphical derivatives. 
\emph{Nonlinear Anal.} 75(3):1324--1340.


\bibitem[{Hoheisel et~al.(2012{\natexlab{b}})}]{Hoheiselerr}
Hoheisel T, Kanzow C, Mordukhovich BS, Phan H (2013) Erratum to ``Generalized Newton's method based on graphical derivatives'' [Nonlinear Anal. TMA 75 (2012) 1324–1340]. \emph{Nonlinear Anal.} 86:157--158.

\bibitem[{Izmailov and Solodov(2014)}]{IS14}
Izmailov A, Solodov M (2014) \emph{Newton-Type Methods for Optimization and Variational Problems} (Springer, New York).

\bibitem[{Izmailov and Solodov(2010)}]{IS10}
Izmailov A, Solodov M (2010) Inexact Josephy–Newton framework for generalized equations and its applications to local analysis of Newtonian methods for constrained optimization. 
\emph{Comput. Optim. Appl.} 46(2):347--368.

\bibitem[{Izmailov and Solodov(2015)}]{Izmailov15}
Izmailov A, Solodov M (2015) Newton-type methods: A broader view. 
\emph{J. Optim. Theory Appl.} 164:577--620.


\bibitem[{Josephy(1979{\natexlab{a}})}]{Josephy1}
Josephy N (1979) \emph{Newton's Method for Generalized Equations And The PIES Energy Model}. 
Ph.D. dissertation (University of Wisconsin-Madison).

\bibitem[{Josephy(1979{\natexlab{b}})}]{Josephy2}
Josephy N (1979) Quasi-Newton method for generalized equations. 
\emph{Technical summary report} ( University of Wisconsin-Madison).

 
\bibitem[{Khanh et~al.(2023{\natexlab{a}})}]{Khanh1}
 Khanh PD, Mordukhovich BS, Phat VT (2023) A generalized Newton method for subgradient systems. 
 \emph{Math. Oper. Res. } 48(4):1811--1845.
 
\bibitem[{Khanh et~al.(2023{\natexlab{b}})}]{Khanh2}
 Khanh PD, Mordukhovich BS, Phat VT, Tran DB (2023) Generalized damped Newton algorithms in nonsmooth optimization via second-order subdifferentials. 
 \emph{J. Global Optim.} 86(1): 93--122.

\bibitem[{Khanh et~al.(2024{\natexlab{a}})}]{Khanh3} 
Khanh PD, Mordukhovich BS, Phat VT (2024) Coderivative-based Newton methods in structured nonconvex and nonsmooth optimization. \emph{ArXiv preprint}, https://arxiv.org/abs/2403.04262.

\bibitem[{Khanh et~al.(2024{\natexlab{b}})}]{Khanh4}
 Khanh PD, Mordukhovich BS, Phat VT, Tran DB (2024) Globally convergent coderivative-based generalized Newton methods in nonsmooth optimization. 
 \emph{Math. Program.} 205(1):373--429.
 

\bibitem[{Klatte and Kummer(2002)}]{KK2}
Klatte D, Kummer B (2002) \emph{Nonsmooth Equations in Optimization}
(Kluwer Academic Publishers, New York).

\bibitem[{Klatte and Kummer(2018)}]{KK18}
Klatte D, Kummer B (2018) 
Approximation and generalized Newton methods. 
\emph{Math. Program.} 168:673--716.


\bibitem[{Kummer(1988)}]{Kummer1}
Kummer B (1988) Newton’s method for non-differentiable functions. 
Guddat J et al., eds. 
\emph{Advances in mathematical optimization }
(De Gruyter, Berlin), 114--125.



\bibitem[{Kummer(1992)}]{Kummer2}
Kummer B (1992) Newton's method based on generalized derivatives for nonsmooth functions: convergence analysis.
Oettli W, Pallaschke D, eds.
\emph{Advances in Optimization}
(Springer, Berlin, Heidelberg), 171--194.


\bibitem[{Kummer(2000)}]{Kumer2000}
Kummer B (2000) Generalized Newton and NCP-methods: Convergence, regularity, actions. 
\emph{Discuss. Math. Differ. Incl. Control Optim.} 20(2):209--244.

\bibitem[{Li et~al.(2018)}]{suitlasso}
Li X, Sun DF, Toh K-C (2018) A highly efficient semismooth Newton augmented Lagrangian method for solving Lasso problems. 
\emph{SIAM J. Optim.} 28(1):433--458.


\bibitem[{Meng et~al.(2005)}]{meng}
Meng F, Sun DF, Zhao G (2005) Semismoothness of solutions to generalized equations and the Moreau-Yosida regularization. 
\emph{Math. Program.} 104:561--581.

\bibitem[{Mifflin(1977)}]{mifflin}
Mifflin R (1977) Semismooth and semiconvex functions in constrained optimization. 
\emph{SIAM J. Control Optim.} 15:957--972.

 

\bibitem[{Mordukhovich(2006)}]{mordukhovich2006variational}
Mordukhovich BS (2006) \emph{Variational Analysis and Generalized Differentiation I: Basic Theory} (Springer, Berlin).

\bibitem[{Mordukhovich (2024}]{mordukhovich2024second}
Mordukhovich BS (2024) \emph{Second-Order Variational Analysis in Optimization, Variational Stability, and Control: Theory, Algorithms, Applications} (Springer, Cham, Switzerland).
 

\bibitem[{Mordukhovich and Sarabi(2021)}]{mordukhovich2021generalized}
 Mordukhovich BS, Sarabi ME (2021) Generalized Newton algorithms for tilt-stable minimizers in nonsmooth optimization. 
 \emph{SIAM J. Optim.} 31(2):1184--1214. 



\bibitem[{Oliveira et~al.(2019)}]{Oliveira19}
de Oliveira F, Ferreira O, Silva G (2019) Newton’s method with feasible inexact projections for solving constrained generalized equations. 
\emph{Comput. Optim. Appl.} 72:159--177.


\bibitem[{Pang(1990)}]{Pang1}
Pang JS (1990) Newton's method for B-differentiable equations. 
\emph{Math. Oper. Res.} 15(2):311--341.



\bibitem[{Pang et~al.(2003)}]{pang2003}
Pang JS, Sun DF, Sun J (2003) Semismooth homeomorphisms and strong stability of semidefinite and Lorentz complementarity problems. 
\emph{Math. Oper. Res.} 28(1):39--63.


\bibitem{QiSunZhou2000}
Qi L, Sun DF, Zhou G (2000) 
A new look at smoothing Newton methods for nonlinear complementarity problems and box constrained variational inequalities.
\emph{Math. Program.} 87:1--35.


\bibitem[{Qi and Sun(1993)}]{qisun}
Qi L, Sun J (1993) A nonsmooth version of Newton's method. 
\emph{Math. Program.} 58:353--367.

\bibitem[{Rademacher(1919)}]{Rademacher}
Rademacher H (1919) \"{U}ber partielle und totale differenzierbarkeit von Funktionen mehrerer Variabeln und \"{u}ber die Transformation der Doppelintegrale.
\emph{Math. Ann.} 79:340-359.

\bibitem[{Rockafellar(1976)}]{rtrmonotone}
Rockafellar RT (1976) Monotone operators and the proximal point algorithm. 
\emph{SIAM J. Control Optim.} 14(5):877--898.


\bibitem[{Rockafellar and Wets(1998)}]{rock1998}
Rockafellar RT, Wets RJB (1998) \emph{Variational Analysis} (Springer, Berlin).

 
\bibitem[{Rockafellar(1970)}]{rock1970}
Rockafellar RT (1970) \emph{Convex Analysis} (Princeton University Press, Princeton).
 
\bibitem[{Robinson(1994)}]{Robinson}
Robinson S (1994) Newton's method for a class of nonsmooth functions. 
\emph{Set-Valued Anal.} 2:291--305.

\bibitem[{Shapiro(2003)}]{shapiro03}
Shapiro A (2003) Sensitivity analysis of generalized equations. 
\emph{J. Math. Sci. (N. Y.)} 115:2554--2565.


\bibitem[{Solodov and Svaiter(2002)}]{solodov02}
Solodov MV, Svaiter BF (2002) A new proximal-based globalization strategy for the Joseph-Newton method for variational inequalities. 
\emph{Optim. Methods Softw.} 17(5):965--983.

\bibitem[{Sun(2006)}]{sunmor}
Sun DF (2006) The strong second-order sufficient condition and constraint nondegeneracy in nonlinear semidefinite programming and their implications. \emph{Math. Oper. Res.} 31(4):761--776.


\bibitem[{Yang et~al.(2015)}]{sdpnalp}
Yang L, Sun DF, Toh K-C (2015) SDPNAL+: a majorized semismooth Newton-CG augmented Lagrangian method for semidefinite programming with nonnegative constraints. 
\emph{Math. Program. Comput.} 7(3):1--36.

\bibitem[{Zhao et~al.(2010)}]{sdpnal}
Zhao X, Sun DF, Toh K-C (2010) A Newton-CG augmented Lagrangian method for semidefinite programming. 
\emph{SIAM J. Optim.} 20(4):1737--1765.
\end{thebibliography}

\end{document}